\documentclass[11pt]{article}
\usepackage{amssymb,amsthm,amsmath,mathrsfs}
\usepackage{prettyref,cite}
\usepackage[colorlinks,linkcolor=blue,citecolor=magenta]{hyperref}
\usepackage{graphicx}
\usepackage{algorithm,algpseudocode}
\usepackage{caption}
\usepackage [a4paper, footskip=1cm, headheight = 16pt, top=3cm, bottom=3cm,  right=2.5cm,  left=2.5cm ]{geometry}
\newcommand\pref[1]{\prettyref{#1}}
\newrefformat{chp}{\chaptername~\ref{#1}}
\newrefformat{sec}{Section~\ref{#1}}
\newrefformat{sub}{Subsection~\ref{#1}}
\newrefformat{app}{Appendix~\ref{#1}}
\newrefformat{sub}{Subsection~\ref{#1}}
\newrefformat{thm}{Theorem~\ref{#1}}
\newrefformat{qst}{Question\ref{#1}}
\newrefformat{lem}{Lemma~\ref{#1}}
\newrefformat{alg}{Algorithm~\ref{#1}}
\newrefformat{cor}{Corollary~\ref{#1}}
\newrefformat{con}{Conjecture~\ref{#1}}
\newrefformat{dfn}{Definition~\ref{#1}}
\newrefformat{rem}{Remark~\ref{#1}}
\newrefformat{fig}{\figurename~\ref{#1}}
\newrefformat{tbl}{\tablename~\ref{#1}}
\newrefformat{eq}{$($\ref{#1}$)$}
\newrefformat{fac}{Fact~\ref{#1}}

\newtheorem{pro}{Proposition}
\newtheorem{cor}[pro]{Corollary}
\newtheorem{lem}[pro]{Lemma}
\newtheorem{thm}[pro]{Theorem}
\newtheorem{con}[pro]{Conjecture}

\newtheorem{prerem}[pro]{Remark}
\newtheorem{predefin}[pro]{Definition}
\newenvironment{rem}{\begin{prerem}\rm}{\hfill $\blacktriangle$\end{prerem}}

\DeclareMathOperator{\lcm}{lcm}

\begin{document}

\title{\textbf{On the packing for triples 
}}
\author{
		Ramin Javadi%
		\thanks{Corresponding author, Department of Mathematical Sciences,
			Isfahan University of Technology,
			P.O. Box: 84156-83111, Isfahan, Iran. School of Mathematics, Institute for Research in Fundamental Sciences (IPM), P.O. Box: 19395-5746,
			Tehran, Iran.  Email Address: \href{mailto:rjavadi@cc.iut.ac.ir}{rjavadi@cc.iut.ac.ir}.}
		\thanks{This research was in part supported by a grant from IPM (No. 95050079).}
		\and 
	Ehsan Poorhadi%
	\thanks{Department of Mathematics,
		Isfahan University of Technology,
		P.O. Box: 84156-83111, Isfahan, Iran. Email Address: \href{mailto:poorhadiehsan70@gmail.com}{poorhadiehsan70@gmail.com}.} 
	\and
		Farshad  Fallah%
	\thanks{Department of Mathemathics,
		Isfahan University of Technology,
		P.O. Box: 84156-83111, Isfahan, Iran. Email Address: \href{mailto:farshad\_1991@yahoo.com}{farshad\_1991@yahoo.com}.} 
}
\date{}
\maketitle
\begin{abstract}
For positive integers $n\geq k\geq t$, a collection $ \mathcal{B} $ of $k$-subsets of an $n$-set $ X $ is called a $t$-packing if every $t$-subset of  $ X $ appears in at most one set in $\mathcal{B}$. In this paper, we give some upper and lower bounds for the maximum size of $3$-packings when $n$ is sufficiently larger than $k$. In one case, the upper and lower bounds are equal, in some cases, they differ by at most an additive constant depending only on $k$ and in one case they differ by a linear bound in $ n $. 
\\
\begin{itemize}
\item[]{{\footnotesize {\bf Key words:}\ Balanced incomplete block design, Triple packing, Group divisible design, Large set.}}
\item[]{ {\footnotesize {\bf Subject classification:} 05B40, 05C70.}}
\end{itemize}
\end{abstract}
\section{Introduction}

Given positive integers  $n,k,t, \lambda $, where $ n\geq k\geq t $, a $ t-(n,k,\lambda) $ \textit{packing} is a pair $ (X,\mathcal{B}) $, where $ X $ is a $ n$-set of  elements (called points) and $ \mathcal{B} $ is a collection of $ k$-subsets of $ X $ (called blocks) such that every $ t$-subset of  $ X $ is contained in at most $ \lambda $ blocks in $ \mathcal{B} $. If each $ t$-subset of  $ X $ appears in exactly $ \lambda $ blocks, then it is called a $ t-(n,k,\lambda) $ \textit{design}. Keevash \cite{keevash}, in a seminal work, establishing a  well-known long-standing conjecture, proved that the obvious necessary conditions for the existence of $t$-designs are also sufficient when $n$ is sufficiently large with respect to $k$ and $\lambda$. When these conditions are not met, we may ask for the maximum size of a $t$-packing. \textit{The packing number} $ D_\lambda(n,k,t) $ is the maximum number of blocks in a $ t-(n,k,\lambda) $ packing.
 A $ t-(n,k,\lambda) $ packing is called \textit{optimal} if it contains exactly $ D_\lambda(n,k,t) $ blocks. 
 In this paper, we focus on the case $ \lambda=1 $ and we write $ D(n,k,t) $ for $ D_1(n,k,t) $. Constructing optimal or near-optimal packings has several applications in different fields such as computer science, cryptography and constant-weight codes in coding theory. 

As an elementary observation, one may verify that $ D(n,k,t) $ satisfies the following recurrent relation (see \cite{johnson}),
\begin{equation} \label{eq:recursion}
D(n,k,t) \leq \min\left\{\left\lfloor \frac{n}{k}\, D(n-1,k-1,t-1)\right\rfloor,\left\lfloor \frac{n}{n-k}\, D(n-1,k,t)\right\rfloor\right\}. 
\end{equation}
By iterating the first part of the above relation, one may obtain the following upper bound for $ D(n,k,t) $ known as Johnson's bound.
\begin{equation} \label{eq:johnson}
D(n,k,t) \leq J(n,k,t) := \left\lfloor \frac{n}{k} \left\lfloor \frac{n-1}{k-1} \cdots \left\lfloor \frac{n-t+1}{k-t+1}\right\rfloor \right \rfloor \right \rfloor.
\end{equation}

It was conjectured by Erd\H{o}s and Hanani \cite{erdoshanani} that the upper bound in \eqref{eq:johnson} is asymptotically tight, in the sense that for every fixed integers $ k,t  $ with $ k\geq t $,
\begin{equation}\label{eq:asym}
\lim_{n\to \infty} \frac{D(n,k,t)}{J(n,k,t)} =1. 
\end{equation}
This conjecture was proved by R\"{o}dl in a breakthrough work \cite{rodl}. His idea of a random construction known as ``nibble" has have a great influence on the development of combinatorics and the similar ideas have been recently applied within the Keevash's proof for the existence of  $t$-designs. In fact, the main result of Keevash \cite{keevash}, where the existence of $t$-designs is one of its consequences, is essentially a probabilistic construction for decomposition of locally dense hypergraphs (or more general structures called complexes) into complete hypergraphs (see \pref{thm:keevash}). In one point of view, his result is a generalization of the Gustavsson's decomposition theorem which asserts that graphs with large minimum degree which satisfy necessary divisibility conditions can be decomposed into complete graphs (cliques).

Proving \eqref{eq:asym} by R\"{o}dl determines the asymptotic behavior of $ D(n,k,t) $ completely. This gives rise to the question that what the exact value of $ D(n,k,t) $ is. This question is of great importance in both theoretical and applied considerations. Since there are some irregularities in small values of $ n $, it is natural to ask for an explicit formula for $ D(n,k,t) $ when $ n $ is sufficiently large. This question was answered for $ t=2 $ by Caro and Yuster \cite{caro,caro2}.  In fact, they proved 
that for $t=2$ and sufficiently large $n$, equality holds in Recursion \eqref{eq:recursion}. In other words, for every positive integer $k$ and sufficiently large integer $n$,

\begin{align} \label{eq:t=2}
 D(n,k,2)=\begin{cases}
\left\lfloor \dfrac{n}{k} \left\lfloor \dfrac{n-1}{k-1}\right\rfloor \right\rfloor & \parbox{6.3cm}{{if either } $(n-1)\not\equiv 0 \pmod {(k-1)}$, {or } $n(n-1)\equiv 0\pmod {(k(k-1))}$,}\\[10pt]
\left\lfloor \dfrac{n(n-1)}{k(k-1)} \right\rfloor-1 & \text{otherwise.}
\end{cases} 
\end{align}

In order to prove this, they mainly work out a suitable minimum ``leave graph'' for each case and decompose its complement (which is dense enough) into cliques using Gustavsson's Theorem. 

The case $t=3$ is much more complicated and a limited number of efforts have been done. The difficulty of the case $t=3$ is due to the fact that we have to construct suitable (possibly small) leave hypergraphs and in some cases such hypergraphs do not exist. This causes the fact that for $t=3$, equality  fails to hold in Recursion~\eqref{eq:recursion}  (see e.g. \eqref{eq:bo1}). Even in some cases, they happen to differ by a lower bound in $\Omega(n)$ (see \pref{thm:p2}).   

In this paper, we are going to investigate the case $ t=3 $ and provide appropriate upper and lower bounds for $ D(n,k,3) $, when $ n $ is large enough (with respect to $ k $). The case $(k,t)=(4,3)$ has been already solved in the literature. In fact, it is proved in \cite{bao} that for every $n\geq 4$,

\[ D(n,4,3)=
\begin{cases}
\left \lfloor \dfrac{n}{4} \left \lfloor \dfrac{n-1}{3} \left \lfloor \dfrac{n-2}{2} \right\rfloor \right\rfloor \right \rfloor & n\not \equiv 0 \pmod 6, \\[10pt]
\left \lfloor \dfrac{n}{4} \left(\left\lfloor \dfrac{(n-1)(n-2)}{6}  \right\rfloor-1\right) \right\rfloor& n\equiv 0 \pmod 6,
\end{cases}
\]
which shows that equality holds in \eqref{eq:recursion} for $(k,t)=(4,3)$.

Our main results can be summarized as follows. For every positive integers $ n,k $, let us define 
	\[J'(n,k,3):=\left\lfloor\frac{n}{k} \left(\left\lfloor \frac{(n-1)(n-2)}{(k-1)(k-2)} \right\rfloor-1\right) \right\rfloor. \]

\begin{thm}\label{thm:main}
Let $ k\geq 4 $ be a fixed integer. There exists a constant $ n_0=n_0(k) $ such that  for every $ n\geq n_0 $,
\begin{itemize}
	\item[\rm (i)]  if $ (n-2)\not\equiv 0 \pmod {k-2} $, then
	 $ D(n,k,3)=J(n,k,3) $. 
	\item[\rm (ii) ]  if $ (n-2)\equiv 0 \pmod {k-2} $, $ (n-1)(n-2)\equiv 0\pmod {(k-1)(k-2)} $ and  $ n(n-1)(n-2)\not\equiv 0 \pmod {k(k-1)(k-2)} $, then 
	\begin{equation} \label{eq:bo1}
	J(n,k,3)-O(k)\leq D(n,k,3)\leq J(n,k,3)-3.
	\end{equation}
	\item[\rm (iii)] In the last case, i.e. when $ (n-2)\equiv 0 \pmod {k-2} $, $ (n-1)(n-2)\equiv  p(k-2)\pmod {(k-1)(k-2)} $, for some integer $p$, $1\leq p\leq k-2$, we have 
\begin{itemize}
	\item If neither $k,p\equiv 4\pmod 6$ nor $k, (p-2)\equiv 0 \pmod 6$, then we have
		\begin{equation} \label{eq:bo2}
J'(n,k,3)-O(k^{3+o_k(1)})\leq  D(n,k,3)\leq J'(n,k,3), 
\end{equation}
	\item otherwise, i.e. if either $k,p\equiv 4\pmod 6$ or $k,(p-2)\equiv 0 \pmod 6$, then we have 
		\begin{equation} \label{eq:bo3}
J'(n,k,3)-O(k^{o_k(1)})n\leq D(n,k,3)\leq J'(n,k,3).
\end{equation}
\end{itemize}
\end{itemize}
Moreover, there are infinitely many $n,k$ such that equalities hold in the upper bounds in \eqref{eq:bo1}, \eqref{eq:bo2} and \eqref{eq:bo3} $($see Remarks~\ref{rem:equality} and \ref{rem:equality2}$)$. In addition, dependency of the lower bound on $n$ in \eqref{eq:bo3} is necessary at least for $p=2$ $($see \pref{thm:p2}$)$.
\end{thm}

These results have been established in two steps. First, using Keevash's decomposition theorem, we prove that the existence of a $3$-packing for large enough $n$ is equivalent to the existence of an appropriate multigraph. Then, using the results on large sets of group divisible designs, we construct such appropriate multigraphs.   

The organization of the paper is as follows. In \pref{sec:tools}, we gather all required preliminaries and tools including Keevash's decomposition theorem and simple GDDs. In \pref{sec:1}, we prove Case (i) of \pref{thm:main}. In \pref{sec:2}, we handle Case (ii) of \pref{thm:main}. In \pref{sec:3}, we delve into Case (iii) of \pref{thm:main}. Finally, in \pref{sec:conc}, we give some concluding remarks for future work and some open problems. 

\section{Preliminaries} \label{sec:tools}
In this section, we explore some definitions and necessary tools that are required for the proof of our main results. 

\subsection{Distinct Decompositions}
A \textit{multigraph} is a pair $ G=(V(G),E(G)) $, where $ V(G) $ is a set of vertices and $ E(G) $ is a multiset of unordered pairs of $ V(G) $ called edges (no loop is allowed). For a pair of distinct vertices $ x,y\in V(G) $, the multiplicity of $ xy $, denoted by $ m_G(xy) $, is the number of appearance of $ xy $ in $ E(G) $ (so, if $x$ and $y$ are nonadjacent, then $m_G(xy)=0$). The multigraph on $n$ vertices where the multiplicities of all pairs of vertices is equal to $\mu$ is denoted by $\mu K_n$.  The \textit{degree} of a vertex  $ x $ in $G $ is defined as $\deg_G(x)= \sum_{y:y\neq x} m_G(xy) $.
For every $x\in V(G)$, the set of vertices which are adjacent to $x$ is denoted by $N_G(x)$ and $N_G[x]$ stands for $N_G(x)\cup \{x\}$.
For a multigraph $ G $, its underlying graph denoted by $ \tilde{G} $, is a simple graph with vertex set $ V(G) $ where $ xy\in E(\tilde{G}) $ if and only if $m_G(xy)\geq 1 $. 

Given a multigraph $ G $ and an integer $ q\geq 3 $, a \textit{distinct $K_q $-decomposition} of $ G $ is a set $ \mathcal{C} $ of distinct $ q$-subsets of $ V(G) $ (called $q$-cliques) such that for every pair of distinct vertices $x,y\in V(G)  $, they are both contained in exactly $ m_G(xy)  $ cliques in $ \mathcal{C} $. 

A \textit{hypergraph} is a pair $G= (V(G),\mathcal{E}(G)) $ where $ V(G) $ is a set of vertices and $ \mathcal{E}(G) $ is a set of subsets of $ V(G) $ called \textit{hyperedges} (or simply \textit{edges}) of $ G $. If all edges of $ G $ have the equal size $ t $, then $G$ is called a  \textit{$ t $-uniform hypergraph} or shortly a \textit{$ t $-graph}. The $ t $-graph $ G $ on $ n $ vertices containing all $ t $-subsets of $V(G)$ as edges is called the complete $ t $-graph and is denoted by $ K_n^{(t)} $.  For every set of vertices $ S\subseteq V(G) $, the degree of $ S $ in $ G $, denoted by $ d_G(S) $, is the number of edges in $ G $ containing $ S $, i.e. 
\[d_G(S)=|\{e\in \mathcal{E}(G) \ : \ S\subseteq e \}|. \]
Given a $ t $-graph $ G $ and an integer $ q> t $, a \textit{$K_q^{(t)}$-decomposition} of $ G $ is a collection $ \mathcal{C} $ of copies of $K_q^{(t)}$ in $ G $ such that every edge of  $ G $ is contained in exactly one of these copies. A $ t $-graph $ G $ is called \textit{$ (q,t) $-divisible} if for every $ 0\leq i\leq t-1 $ and every subset $ S\subseteq V(G) $ with $ |S|=i $, $ \dbinom{q-i}{t-i} $ divides $ d_G(S) $. The obvious necessary condition for $ G $ to have a $ K_q^{(t)} $-decomposition is that $ G $ should be $ (q,t) $-divisible.

The complement of a $t$-graph $ G $, denoted by $ \overline{G} $, is a $ t$-graph  on the same set of vertices $ V(G) $ such that for every $ t$-subset $ e $ of $ V(G) $, we have $ e\in \mathcal{E}(\overline{G}) $ if and only if $ e\not\in \mathcal{E}(G)  $. 

The following theorem is the minimum degree version of the Keevash's decomposition theorem which can decompose the hypergraphs with large minimum degree (see Theorem~1.4 in \cite{keevash}, also see Theorem~1.1 in \cite{kuhn}). 
\begin{thm}\label{thm:keevash}{\rm\cite{keevash}}
For all integers $ q>t\geq 2 $, there exists integer $ n_0$ and number $c>0 $ such that the following holds for every $ n\geq n_0 $. 
Let $ G $ be a $ t$-graph on $ n $ vertices such that for every $ S\subseteq V(G) $, with $ |S|=t-1 $, we have $ d_G(S)\geq (1-c)n $. If $ G $ is $ (q,t)$-divisible, then $ G $ admits a $ K_q^{(t)}$-decomposition.   
\end{thm}
The Keevash's decomposition theorem can also be applied to decompose multigraphs and multi-hypergraphs in certain conditions. The following is a corollary of Theorem 6.5 in \cite{keevash}. 

\begin{lem}\label{lem:multigraph}{\rm\cite{keevash}}
Let $ q, \lambda$ be positive integers where  $ q\geq 3 $. Then, there exist constants $ n_0 $ and $ c $  (depending on $ q,\lambda $) such that the following holds for all $ n\geq n_0 $. Let $G$ be a multigraph on $n$ vertices satisfying the following conditions,
\begin{itemize}
	\item $ G $ is $ (q,2)$-divisible, i.e. $ \dbinom{q}{2} $ divides $ |E(G)| $ and $ q-1 $ divides the degree of all vertices,
	\item for every pair of vertices $ x,y $, we have $ m_G(xy)\in \{0,\ldots, \lambda\} $, and for every vertex $x$, $ |\{y\ : \ m_G(xy)< \lambda \}|\leq cn $.
\end{itemize}
Then, $ G $ admits a distinct $ K_q $-decomposition.
\end{lem}
In the above lemma,  the second condition says that a small fraction of edges on each vertex could have multiplicity less than  $ \lambda $. However, we need to find decompositions for some multigraphs for which the dominant multiplicity is not necessarily the largest one. Thus, we prove a slight extension of \pref{lem:multigraph} as follows. 
\begin{lem}\label{lem:multigraph2}
	Let $ q, \lambda$ and $ \lambda' $ be positive integers where  $ q\geq 3 $ and $ \lambda\leq \lambda' $. Then, there exist constants $ n_0 $ and $ c $  (depending on $ q,\lambda' $) such that the following holds for all $ n\geq n_0 $. Let $G$ be a multigraph on $n$ vertices satisfying the following conditions,
	\begin{itemize}
		\item $ G $ is $ (q,2)$-divisible, and
		\item for every pair of vertices $ x,y $, we have $ m_G(xy)\in \{0,\ldots, \lambda'\} $, and for every vertex $x$, $ |\{y\ : \ m_G(xy)\neq \lambda \}|\leq cn $.
	\end{itemize}
	Then, $ G $ admits a distinct $ K_q $-decomposition.
\end{lem}
\begin{proof}
Fix a number $ c $ and let $ G $ be a multigraph satisfying the above conditions. We are going to remove some distinct $ q$-cliques from $ G $ such that the remained multigraph fulfills the conditions of \pref{lem:multigraph}. For this, we have to reduce the multiplicities which are larger than $ \lambda $. For every vertex $ x $, define
\[\Gamma_x=\{y\ :\ m_G(xy)>\lambda \}, \]
where by definition, we have $ |\Gamma_x|\leq cn $.

Consider an ordering of $ V(G) $, as $ x_1,\ldots, x_n $ and sequentially do the following for each $ 1\leq i\leq n  $,
\begin{itemize}
	\item \textbf{Step $ i $.} For every vertex $ x\in \Gamma_{x_i} $, if there is an edge between $ x $ and $ x_i $ in $ G$, choose additional vertices $ y_1,\ldots, y_{q-2} $ consecutively to form a $ q$-clique $ K=\{x_i,x,y_1,\ldots, y_{q-2}\} $, as follows. Assume that the vertices $ y_1,\ldots, y_{j-1} $ are chosen and we are going to choose $ y_j $. We say that a vertex $ y $ is \textit{valid} if $ y $ is adjacent to all vertices $ x_i,x, y_1,\ldots,y_{j-1} $ and also when $ j=q-2 $, the clique $ \{x_i,x, y_1,\ldots,y_{j-1},y\} $ is not chosen before. Now, choose a valid vertex $ y_j $ whose number of appearance in previously chosen cliques is minimal among all valid vertices. When all vertices $ y_1,\ldots,y_{q-2} $ are chosen, remove all edges of the $ q$-clique $ K $ from $ G $ (in fact, the multiplicities of the pairs in $ K $ are reduced by one). Iterate this until there is no edge between $ x $ and $ x_i $ in $ G $.
\end{itemize}

We are going to prove that the above procedure continues until all edges between the vertex $ x_i $ and the vertices in $ \Gamma_{x_i} $ are removed, for all $ 1\leq i\leq n $. 
Suppose that in step $ i $, the vertices $ y_1,\ldots, y_{j-1} $ are chosen and we are going to choose the vertex $ y_j $.  Fix a vertex $ u$ in $L= \{x_i,x,y_1,\ldots, y_{j-1}\} $ and suppose that $ Y $ is the set of all vertices $ y $ in $ V(G)\setminus L $ which are nonadjacent to $ u $ (and so are not valid). 
We are going to prove that $ |Y|\leq c'cn $ for some constant $ c'=c'(q,\lambda') $.

Let $ Y' $ be the set of all vertices $ y $ in $ Y $ such that either $ y$ is nonadjacent to $ u $ in the original graph $ G $, or  the last edge between $ y $ and $ u $ is removed in some step $ i' $ and $ u\in  \{x_{i'}\}\cup \Gamma_{x_{i'}} $. Then, we have $ |Y'|\leq cn+ \lambda'cn(q-1) $ (because $ |\Gamma_{u}|\leq cn $ and for every $ u'\in \Gamma_{u} $, $ uu' $ is contained in at most $ \lambda' $ chosen cliques and each of these cliques contains $q-1$ vertices except $ u $).   

Now, let $ Y''=Y\setminus Y' $. For every $ y\in Y'' $, let $ K^y $ be the last chosen $ q$-clique containing $ u$ and $y $ (they could be the same for some $ y $'s). Note that,  since $ y\not\in Y' $, $ u $ is added to $ K^y $ as additional vertices. For every vertex $ w\in V(G) $, let $ Y_w $ be the set of all vertices $ y\in Y'' $ for which $ w $ is valid at the time that $ u $ is adding to $ K^y $. Fix a vertex $ w\in V(G) $ and let $ y_0\in Y_w$ be such that $ K^{y_0} $ is the last chosen clique among the cliques $ K^y, y\in Y_w $. At the time that  $ u $ is adding to $ K^{y_0} $, $ w $ is valid and $ u $ has already been chosen at least $ (|Y_w|-1)/(q-1)$ times (because $ u $ is contained in all cliques $ K^y $, $ y\in Y_w\setminus \{y_0\} $ and these cliques are the same for at most $ (q-1) $ vertices $ y $). Thus, $ w $ has also been chosen at least $ (|Y_w|-1)/(q-1)$ times (since $ u $ is chosen instead of $ w $). Moreover, let $ \overline{Y_w}= Y''\setminus Y_w $. For every vertex $ y\in \overline{Y_w} $, $ w $ is nonadjacent to a vertex in $ K^y $. Thus, there exist at least $ |\overline{Y_w}|/(\lambda'(q-1))-cn  $ vertices nonadjacent to $ w $ which are adjacent to $ w $ in the original graph $ G $ (because $ w $ is nonadjacent to at most $ cn $ vertices in the original $ G $, the cliques $ K^y $ are the same for at most $ (q-1) $ vertices $ y $ and every vertex except $ u $ appears in at most $ \lambda' $ distinct cliques $ K^y $). This shows that $ w $ is contained in at least $ |\overline{Y_w}|/(\lambda'(q-1)^2)-cn/(q-1)  $ chosen cliques. 
These two facts imply that every vertex appears in at least $ c''(|Y''|-cn) $ chosen cliques for some constant $ c''=c''(q,\lambda') $. Thus, there are at least $ c''(|Y''|n-cn^2)/q $ chosen cliques. On the other hand, at the end of the procedure, at most $ \lambda'cn^2 $ cliques are chosen (because $ |\Gamma_x|\leq cn $ for all vertices). Therefore, $ |Y''|\leq  cn(\lambda'q+c'')/c''$. This implies that $ |Y|\leq c'cn $ for some constant $ c'=c'(q,\lambda') $. Also, in the case that $ j=q-2 $, there are at most $ \lambda' $ vertices $ y $ such that the clique $ L\cup \{y\} $ has been chosen before. Thus, in each step, there exist at least $ (1-qc'c)n-\lambda' $ valid vertices. So, as long as $ n\geq n_0 $ for some constant $ n_0 $, the procedure does not stop until step $n $.  

Let  $ G'$ be the multigraph obtained from $G $ by performing the above procedure. It is clear that $ m_{G'} (xy)\leq \lambda $ for all pairs $ x,y $. Also, since $ G $ is $ (q,2)$-divisible and $ G' $ is obtained from $ G $ by removing some $ q$-cliques of $ G $, $ G' $ is also $ (q,2)$-divisible. Also, note that in the above procedure, the degree of every vertex is reduced by at most $ \hat{c} cn $, for some constant $ \hat{c} $. To see this, let $ x $ be a vertex in $ G $ and assume that $ x $ is appeared in $ l $ chosen cliques. So, $ x $ appears in at least $ l-\lambda' cn $ cliques as an additional vertex. At the time that $ x $ is added to the last clique, there are at least $ (1-qc'c)n-\lambda' $ valid vertices whose number of appearance in previously chosen cliques is at least $ l-\lambda' cn $. Therefore, the number of chosen cliques is at least $ (l-\lambda' cn)((1-qc'c)n-\lambda')/q $. On the other hand, the number of chosen cliques is at most $ \lambda'cn^2 $. Thus, $l\leq \hat{c} cn $, for some constant $ \hat{c} $. Hence, $ |\{y\ :\ m_{G'}(xy)<\lambda\}|\leq (1+\hat{c})cn $. Consequently, $ G' $ fulfills the conditions of \pref{lem:multigraph} and admits a distinct $K_q$-decomposition. This decomposition along with the chosen cliques in the above procedure comprise a  distinct $K_q$-decomposition for $ G $. (Note that the decomposition of $ G' $ does not contain any of the cliques chosen by the procedure, because there is no edge between every vertex $ x $ and a vertex in $ \Gamma_x $ in $ G' $.)
\end{proof}
We also need the following theorem about the decomposition of complete multigraphs into triangles.
\begin{thm}\label{thm:triangle}{\rm \cite{dehon}}
Suppose that $ n,\lambda $ are positive integers. Then, $ \lambda K_n $ admits a distinct $ K_3 $-decomposition if and only if $ \lambda(n-1) $ is even, $ \lambda n(n-1) $ is divided by $ 3 $ and $ \lambda \leq n-2 $.
\end{thm}

\subsection{Connection of Packing and Decomposition}
Here, we show how constructing a suitable leave multigraph and using \pref{thm:keevash} can lead to the construction of an appropriate $3$-packing.

\begin{lem}\label{lem:multigraph3}
For every positive integers $ n,k,\xi$, if $ D(n,k,3)\geq \xi$, then there exists a multigraph $ G $ satisfying the following conditions. 
\begin{itemize}
	\item[\rm (i)] $|V(G)|=n$ and $2|E(G)|= n(n-1)(n-2)-k(k-1)(k-2) \xi$. 
	\item[\rm (ii)] For every vertex $ x\in V(G) $, $\deg_G(x)=(n-1)(n-2) \pmod {(k-1)(k-2)}$.
	\item[\rm (iii)] For every pair of vertices $ x,y\in V(G) $, $ m_G(xy)= n-2 \pmod{k-2} $.
	\item[\rm (iv)] The graph $ G $ has a distinct $ K_3$-decomposition. 
\end{itemize}
In addition, for every positive integers $ k,\sigma $, there exists integer $ n_0=n_0(k,\sigma) $ such that for every integer $ n\geq n_0 $, if there exists a multigraph $ G $ satisfying Conditions (i)-(iv), as well as the following condition
\begin{itemize}
	\item[\rm (v)] For every pair of vertices $ x,y\in V(G) $, $ m_G(xy)\leq \sigma$,
\end{itemize}
 then $ D(n,k,3)\geq \xi$.
\end{lem}

\begin{proof}
Suppose that $ D(n,k,3)\geq \xi$. So, there exists a $ 3-(n,k,1) $ packing with $ \xi$ blocks. One may consider such a packing as a $ 3$-graph $ H $ on $ n $ vertices whose edges are triples which appear in the blocks of the packing. Also, $ H $ admits a $ K_k^{(3)} $-decomposition. So, $ H $ is $ (k,3)$-divisible. Now, the multigraph $ G $ on $ n $ vertices is obtained from $ \overline{H} $, the complement of $ H $, as follows. We set $ V(G)=V(H) $ and for each pair of vertices $ x,y\in V(G) $, $ m_G(xy) $ is equal to $ \deg_{\overline{H}}(\{x,y\}) $. Now, we check that Conditions (i)-(iv) holds.

First, since $ H $ is $ (k,3)$-divisible, for every pair of vertices $ x,y\in V(G) $, we have $ m_G(xy)= d_{\overline{H}} (\{x,y\})= n-2-d_{H} (\{x,y\})=n-2 \pmod{k-2}  $. Thus, Condition (iii) holds. Also, we have $ \deg_G(x)= 2\deg_{\overline{H}}(\{x\}) $, so Condition (ii) follows from $ (k,3) $-divisibility of $ H $ in a similar way. Also, Condition (iv) holds because $ E(\overline{H}) $ can be considered as a distinct $ K_3 $-decomposition of $ G $. Finally, note that $ |E(G)|=3|E(\overline{H})|= 3 (\binom{n}{3}- |E(H)|)= 3 (\binom{n}{3}- \xi\binom{k}{3}) $. Thus, Condition (i) holds.

Now, we prove the converse. Suppose that $ G $ is a multigraph satisfying Conditions (i)-(v). Consider the $ K_3 $-distinct decomposition of $ G $ as a $ 3$-graph on the vertex set $ V(G) $ and let $ H $ be its complement. It is clear that $ H $ is $ (k,3)$-divisible. Let $ n_0 $ and $ c $ be the numbers in \pref{thm:keevash}. Due to Condition (v), for every pair of vertices $ x,y $, we have $\deg_H(\{x,y\}) \geq n-2-\sigma  $. Now, let $ n_1 $ be an integer such that $ cn_1\geq \sigma+2  $. So, for every $ n\geq n_1 $, we have  $\deg_H(\{x,y\}) \geq (1-c)n  $. Hence, for every $ n\geq \max\{n_0,n_1\} $, $ H $ fulfills the conditions of \pref{thm:keevash} and thus has a  $ K_k^{(3)}$-decomposition. This forms a $3-(n,k,1)$ packing of size $ |E(H)|/\binom{k}{3}=\xi$. Therefore, $ D(n,k,3)\geq \xi$.
\end{proof}

\subsection{Simple GDDs}
One of the important combinatorial notions that has been used in this paper is the group divisible design.  
Let $v, k, \lambda $ be positive integers and $ U $ be a set of positive integers. A \textit{group divisible
design} of index $ \lambda $ and order $ v $ denoted by $(k,\lambda)$-GDD, is a triple $(X, \mathcal{G}, \mathcal{B})$, where $ X $ is a set of size $ v $, $ \mathcal{G} $ is a partition of $ X $ into groups whose sizes lie in $ U $,  where $ |\mathcal{G}|>1 $ and $ \mathcal{B} $ is a family of $ k$-subsets (blocks) of $ X $ such that every pair of distinct elements of $ X $ occurs in exactly $ \lambda $ blocks or one group, but not both.

When $ U=\{g_1,\ldots, g_s\} $ and there are exactly $ u_i $ groups of size $ g_i $, $ 1\leq i\leq s $, where $ v=\sum_{i=1}^s g_iu_i $, we say that $ (k,\lambda) $-GDD is of type $ g_1^{u_1}\ldots g_s^{u_s} $ and for simplicity we write $ (k,\lambda) $-GDD$ (g_1^{u_1}\ldots g_s^{u_s} ) $.
 A GDD is said to be \textit{simple} if no two blocks are identical. 

The following result determines the necessary and sufficient conditions for the existence of a simple $ (3,\lambda) $-GDD$ (1^u) $ i.e. a triple system. 
\begin{thm} {\rm \cite{dehon}} \label{thm:simpleTS}
There exists a simple $ (3,\lambda) $-GDD$ (1^u) $ if and only if $ 1\leq \lambda \leq u-2 $, $ \lambda(u-1) \equiv 0 \pmod 2 $ and $ \lambda u(u-1) \equiv 0 \pmod 6 $.
\end{thm}
Two $(k,\lambda)$-GDD$(g^u)$s are said to be disjoint if they have the same set of groups and their
block sets are disjoint.  A \textit{large set} of GDDs denoted by  $(k,\lambda)$-LGDD($ g^u $) is a collection $(X, \mathcal{G}, \mathcal{B}_i)_{i\in I}$ of pairwise disjoint simple $(k,\lambda)$-GDD$(g^u)$s,
all having the same set of groups, such that for every $k$-subset $ B $ of $ X $, if $ |B\cap G|\leq 1 $,
for all $G \in \mathcal{G}$, then there is exactly one $i\in I$ where $ B\in \mathcal{B}_i $. 
The following theorem gives the necessary and sufficient condition for the existence of a $(3,\lambda)$-LGDD$(g^u)$.
\begin{thm} {\rm \cite{lei,chini}} \label{thm:largeset}
There exists a $(3,\lambda)$-LGDD($g^u$) if and only if $\lambda g(u - 1) \equiv 0 \pmod 2$, $\lambda g^2u(u - 1) \equiv 0 \pmod 3$, $ gu(u-2) \equiv 0\pmod \lambda $, $u \geq  3$ and $(\lambda,g, u) \neq (1,1, 7)$.
\end{thm}
It is easy to check that a $(k,\lambda)$-LGDD$(g^u)$ consists of  $ \binom{u-2}{k-2} g^{k-2}/\lambda $ pairwise disjoint simple $(k, \lambda)$-GDD($ g^u $)s. Therefore, if  a $(k,\lambda)$-LGDD$(g^u)$ exists, then for every $ t $, $ 1\leq t\leq \binom{u-2}{k-2} g^{k-2}/\lambda  $  one may construct a simple $(k,\lambda t) $-GDD$ (g^u) $ by juxtaposition of $ t $ pairwise disjoint $(k, \lambda)$-GDD($ g^u $)s. Hence, the following is a consequence of \pref{thm:largeset} (the case $ (g,u)=(1,7) $ follows from \pref{thm:simpleTS}.  For more detail, see \cite{chini}). 
\begin{cor} \label{cor:simpleGDD}
	A simple $(3,\lambda) $-GDD$ (g^u) $ exists if and only if  $\lambda g(u - 1) \equiv 0 \pmod 2$, $\lambda g^2u(u - 1) \equiv 0 \pmod 3$, $u \geq  3$ and $1\leq \lambda \leq g(u-2) $.
\end{cor}
%
%
%
%
\subsection{System of linear Diophantine equations}
Finally, we will need the following simple result which can be viewed as a generalization of Chinese Remainder Theorem.
\begin{lem}\label{lem:number}
Let $ p_1,\ldots, p_s, q_1,\ldots, q_t $ be some prime powers with distinct bases such that $ q_j\geq 4 $, for all $ 1\leq j\leq t $. Let $ M= p_1\cdots p_s $ and $ N=q_1\cdots q_t $. Also, suppose that $ a_1,\ldots, a_s$, $b_1,\ldots, b_t$, $c_1,\ldots, c_t$, $d_1,\ldots,d_t $ are some integers. Then, there exists a positive integer $ x $, where $ x\leq N^{1/\Omega(\log\log N)} M $ such that
\begin{align}
x & \equiv a_i \pmod {p_i}, \ \forall\ 1\leq i\leq s,\label{eq:eq} \\
x & \not \equiv b_i,c_i,d_i \pmod {q_i}, \ \forall\ 1\leq i\leq t.\label{eq:neq}
\end{align}
\end{lem}
\begin{proof}
Define $ k=\min \{i : q_{i+1}\geq 3t+1 \} $ and let $ N'=(p_1\cdots p_s)(q_1\cdots q_k) $. Choose integers $ e_i $, $ 1\leq i\leq k $, such that $ e_i\not \equiv b_i,c_i,d_i \pmod {q_{i}} $ which exist since $ q_i\geq 4 $. By Chinese Remainder Theorem, there exists integer $ r $, $ 1\leq r\leq N' $, such that 
\begin{align*}
r & \equiv a_i \pmod {p_i}, \ \forall\ 1\leq i\leq s,  \\
r & \equiv e_i \pmod {q_i}, \ \forall\ 1\leq i\leq k. 
\end{align*}
Now, for each $ i $, $ k+1\leq i\leq t $, let $ \hat{b}_i,\hat{c}_i,\hat{d}_i $ be three integers such that $1\leq  \hat{b}_i,\hat{c}_i,\hat{d}_i\leq q_i $ and $ \hat{b}_iN' \equiv b_i-r \pmod{q_i}  $, $ \hat{c}_iN' \equiv c_i-r \pmod{q_i}  $ and $ \hat{d}_iN' \equiv d_i-r \pmod{q_i}  $. Such integers exist since $ N' $ and $ q_i $ are coprime. Now, choose an integer $ h $ in the set $ \{1,2,\ldots, 3t+1 \}\setminus \{\hat{b}_1,\ldots,\hat{b}_t,\hat{c}_1,\ldots,\hat{c}_t,\hat{d}_1,\ldots,\hat{d}_t \} $. Since $ q_i\geq 3t+1 $, for $k+1\leq i\leq t $, it is clear that 
\[
h \not \equiv \hat{b}_i,\hat{c}_i,\hat{d}_i \pmod {q_i}, \ \forall\ k+1\leq i\leq t.
\]
Therefore, 
\[
N'h+r \not \equiv b_i,c_i,d_i \pmod {q_i}, \ \forall\ k+1\leq i\leq t.
\]
Hence, $ x=N'h+r $ satisfies \eqref{eq:eq} and \eqref{eq:neq}. Also,
we have $ x\leq N'(3t+2)= M(3t+2)q_1\cdots q_k $. By well-known estimates of common number theoretical functions \cite{numbert}, we have $ q_1\cdots q_k \leq 2^{(1+o(1))3t} $ and also, $ N\geq 2^{\Omega(t\log t)} $. Thus, there is a constant $c$ such that $x\leq M2^{ct}$. Now, if  $ t \leq \sqrt{\log N } $, then $ x\leq M2^{ct}\leq M 2 ^ {c\sqrt{\log N}}\leq MN^{1/\Omega(\log\log N)}  $. Finally, if $ t \geq \sqrt{\log N }$, then $ x\leq M 2^{ct} \leq M N^{1/\Omega(\log t)} \leq M N^{1/\Omega(\log \log N)} $.  
\end{proof}

\section{The case $(n-2) \not\equiv 0 \pmod{k-2}$\label{sec:1}} 
In this section, we prove that if $ n $ is large enough and $ (n-2)\not\equiv 0 \pmod{k-2} $, then equality holds in \eqref{eq:johnson}. 

\begin{thm} \label{thm:r!0}
Let $ k\geq 3 $ be a fixed integer. There exists a constant $ n_0=n_0(k) $ such that  for every $ n\geq n_0 $, if $ (n-2)\not\equiv 0 \pmod{k-2} $, then $ D(n,k,3)=J(n,k,3) $.
\end{thm}

In order to prove the above theorem, we need the following lemma about the existence of simple graphs which is a straightforward corollary of Erd\H{o}s-Gallai Theorem \cite{erdosgallai}.

\begin{lem}\label{lem:erdosgallai}
	For every positive integer $ c $, there is a constant $ n_0=n_0(c) $, such that for every $ n\geq n_0 $ and every integers $d_1, \ldots, d_n  $, if $ 0\leq d_i\leq c $ and $ \sum d_i $ is even, then there exists a simple graph on $ n $ vertices with the degree sequence $ (d_1,\ldots, d_n) $.
\end{lem}

\begin{proof}[\rm \textbf{Proof of \pref{thm:r!0}}]
The upper bound is obtained from \eqref{eq:johnson}. Now, we prove the lower bound. Let $ r\equiv (n-2) \pmod{k-2} $, where $ 0<r<k-2 $. Also, let $ \alpha\equiv (n-1)(n-r-2) \pmod {(k-1)(k-2)} $ and $ \beta \equiv (n(n-1) (n-2-r)-n\alpha) \pmod {k(k-1)(k-2)} $, where  $ 0\leq \alpha< (k-1)(k-2) $ and $ 0\leq \beta< k(k-1)(k-2) $. By these notations, we have 
\[J(n,k,3)= \frac{n(n-1)(n-2-r) - \alpha n- \beta}{k(k-1)(k-2)}. \]

By \pref{lem:multigraph3}, it is enough to prove that there exists a multigraph $ G $ on $ n $ vertices satisfying Conditions (i)-(v) such that $ |E(G)|= (rn(n-1)+\alpha n+\beta)/2$. We construct the multigraph $ G $ as follows.

%
%
%
%
%
%
For this, first define
\begin{equation}  \label{eq:gammax}
\gamma_0= \frac{\alpha+\beta}{k-2} \text{ and } \gamma=\frac{\alpha}{k-2}.
\end{equation}
%
%
%

By definitions of $\alpha$ and $\beta$, we have $\gamma_0$ and $\gamma$ are integers. Also, note that by \pref{lem:erdosgallai}, if $ n $ is sufficiently large (with respect to $ k $), there exists a simple graph $ G' $ on $ n $ vertices with one vertex $ x_0 $ of degree $ \gamma_0 $ and other vertices of degree $ \gamma $ (because the degrees sum is equal to $ (n-1)\gamma+\gamma_0=(n\alpha +\beta)/(k-2)  $ which is even by the definition of $ \beta $ and also $ \gamma\leq k $ and $ \gamma_0\leq k^2 $).

Let $ G $ be the edge disjoint union of $ (k-2)G' $ and $ rK_n $ on the same set of vertices. (For a graph $ H $ and integer $ \lambda $, the graph  $\lambda H  $ is obtained from $ H $ by replacing each edge by $ \lambda $ parallel edges). It is clear that $ 2|E(G)|=rn(n-1)+\alpha n+\beta $. Also, the multiplicity of each edge is equal to either $ r $ or $ r+k-2 $. So, Conditions (iii) and (v) hold.  Moreover, for every vertex $ x\neq x_0 $, we have $ \deg_G(x)=r(n-1)+\alpha $ and also $ \deg_G(x_0)=r(n-1)+\alpha+\beta $. By definitions of $ \alpha  $ and $ \beta $, we have $ r(n-1)+\alpha \equiv (n-1)(n-2) \pmod{(k-1)(k-2)} $ and $ \beta\equiv 0 \pmod{(k-1)(k-2)} $. Thus, Condition (ii) holds. Finally, it remains to prove that Condition (iv) holds, i.e. $ G $ admits a distinct $ K_3 $-decomposition for sufficiently large $ n $.

%
%

For this, we apply \pref{lem:multigraph2} by setting $ q=3 $, $ \lambda'=r+k-2 $ and $\lambda=r  $. For every vertex $ x $, the number of vertices $ y $ such that $ m_{G}(xy)\neq r $ is at most $ \gamma_0\leq k^2 $. Also, by Condition (ii), the degree of every vertex is even and $ |E(G)| $ is divisible by 3 for $2|E(G)|= rn(n-1)+\alpha n+\beta \equiv n(n-1)(n-2)\pmod {k(k-1)(k-2)} $. Thus, $G $ is $ (3,2) $-divisible and, by \pref{lem:multigraph2}, admits a distinct $ K_3$-decomposition when $ n $ is sufficiently large. This completes the proof. 
%
%
%
\end{proof}

\section{The case $ (n-2)\equiv 0 \pmod {k-2} $ and $ (n-1)(n-2)\equiv 0\pmod {(k-1)(k-2)} $\label{sec:2}}
Now, we go through the case that $ (n-2)\equiv 0 \pmod{k-2} $ and $ (n-1)(n-2)\equiv 0\pmod {(k-1)(k-2)} $. If, in addition, $ n(n-1)(n-2)\equiv 0 \pmod {k(k-1)(k-2)} $, then the necessary conditions for the existence of a design hold and by the result of Keevash \cite{keevash}, we have $ D(n,k,3)=J(n,k,3) $, for sufficiently large $ n $.
In the following theorem, we look into the other case when  $n(n-1)(n-2)\not\equiv 0 \pmod {k(k-1)(k-2)}$.
\begin{thm} \label{thm:upper_r0a0b!0}
	Let $ n,k $ be two positive integers such that $ n>k\geq 4 $, $ (n-2)\equiv 0 \pmod{k-2} $, $ (n-1)(n-2)\equiv 0\pmod {(k-1)(k-2)} $ and  $ n(n-1)(n-2)\not\equiv 0 \pmod {k(k-1)(k-2)} $. Then, $ D(n,k,3)\leq J(n,k,3)-3 $.
\end{thm}
At the end of this section, we will show that there are infinitely many numbers $ n,k $ such that equality holds in the above inequality (see \pref{rem:equality}).

\begin{proof}
	Let $ n,k $  be such that $ (n-2)\equiv 0 \pmod {k-2} $ and  $ (n-1)(n-2)\equiv 0\pmod {(k-1)(k-2)} $. Also, let $ \beta \equiv n(n-1) (n-2) \pmod {k(k-1)(k-2)} $, where $ 0< \beta< k(k-1)(k-2) $. Thus, $ \beta $ is a multiple of $ (k-1)(k-2) $. Let $ \beta =q (k-1)(k-2) $ for some $ 0<q<k $. Also, we have 
	\[J(n,k,3)= \frac{n(n-1)(n-2) - \beta}{k(k-1)(k-2)}. \]

	Now, for the contrary, suppose that $ D(n,k,3)\geq J(n,k,3)-2 $. So, by \pref{lem:multigraph3}, there exists a multigraph $ G $ on $ n $ vertices satisfying the following conditions.

%
	\begin{itemize}
		\item[(i)] $ 2|E( G )|=\beta+2k(k-1)(k-2) $,
		\item[(ii)] $\forall\ x\in V( G )$,\  $  d_{G}(x) \equiv 0 \pmod {(k-1)(k-2)}  $, 
		\item[(iii)] $ \forall x\neq y\in V( G ) $, $ m_{ G }(xy) \equiv 0 \pmod {k-2} $, and
		\item[(iv)] $ G $ admits a distinct $ K_3 $-decomposition. 
	\end{itemize}
	Since $|E( G )|$ does not depends on $ n $, most of the vertices are isolated. Now, remove all isolated vertices and by abuse of notation, denote the obtained multigraph by $ G $.  
	In the rest of the proof, we prove that no such multigraph $ G $ exists and this contradiction implies the assertion. Assuming the existence of $ G $, let $ \nu $ (resp. $ \mu $) be the number of vertices of $ G $ with degree $ 2(k-1)(k-2) $ (resp. at least $ 3(k-1)(k-2) $). Then, by Conditions (i) and  (ii) and taking the degree sum, we have 
	\begin{equation}\label{eq:vG}
	|V(G)|\leq q+2k-\nu-2\mu.
	\end{equation}
	Moreover, we frequently use the following fact which is a straightforward consequence of Condition (iv).\\
	
	\textbf{Fact 1.} For every $ x, y\in V( G ) $, if $ m_{G}(xy)=t $, then $ x $ and $ y $ have at least $ t $ common neighbors. \\
	
	We prove the claim in three cases. \\ 
	
	\textbf{Case 1.} The multiplicity of each (existing) edge is equal to $ k-2 $.
	
	First, note that there are at least two adjacent vertices of degree at least $ 2(k-1)(k-2) $. For this, suppose the contrary and let $ x $ be a vertex such that all its neighbors are of degree $ (k-1)(k-2) $. By Fact 1, for every $ y\in N_G(x) $, $ x $ and $ y $ have $ k-2 $ common neighbors and $ y $ has no other neighbor. Thus, the connected component of $ G $ containing $ x $ is the union of copies of $ (k-2)K_k $ sharing the vertex $ x $. Since there is such vertex $x$  in each connected component of $G$, we conclude that all blocks of $ G $ are isomorphic to $ (k-2)K_k $ which is in contradiction with Condition (i) (since $ \beta\neq 0 $). This proves that there are at least two adjacent vertices of degree at least $ 2(k-1)(k-2) $.
	
	Now, Let $ x_0 $ be a vertex of degree at least $ 2(k-1)(k-2) $. Also, let $ X $ be the set of neighbors of $ x_0 $ with degree at least $ 2(k-1)(k-2) $ and let $ Y $ be the set of other neighbors of $ x_0 $ (so, by the above argument, $ X\neq \emptyset $). Then, all neighbors of every vertex in $ Y $ are in $ X\cup Y\cup \{x_0\} $ (because $ y\in Y $ and $ x_0 $ have $ k-2 $ common neighbors and so $ N_{G}(y)\subseteq N_{G}[x_0] $). 
	Also, every vertex in $ X $ has at most $ (k-2) $ neighbors in $ Y $. To see this, let $ x\in X$ be adjacent to $ y\in Y $. Then, the triangle $ \{x_0,x,y\} $ is contained in any distinct $ K_3 $-decomposition of $ G $ (because $ x_0 $ and $ y $ have exactly $ k-2 $ common neighbors including $ x $). On the other hand, the multiplicity of $ x_0x $ is equal to $ k-2 $. So, there are at most $ k-2 $ vertices in $ Y $ adjacent to $ x $. Therefore, every vertex $ x\in X $ has at most $ k-3+|X| $ neighbors in $ X\cup Y $. Thus, $ x $ has at least $ 2(k-1)-(k-3+|X|)-1= k-|X| $ neighbors in $ V(G)\setminus (X\cup Y\cup \{x_0\}) $. Hence, $ |V(G)|\geq 2k-1+k-|X|=3k-|X|-1 $.  Since $\nu+\mu \geq |X|+1 $ and $ q\leq k-1 $, we have $ |V(G)|\geq 3k-|X|-1 \geq 2k+q-\nu-\mu+1  $ which is in contradiction with \eqref{eq:vG}.\\
	
	\textbf{Case 2.} There is an edge with multiplicity at least $3 (k-2) $.
	
	Let $ xy $ be an edge of $ G $ with multiplicity at least $ 3(k-2) $. So, due to Fact 1, $ |V(G)|\geq 3(k-2)+2=3k-4 $ and $ x $ and $ y $ have at least $ 3(k-2) $ common neighbors. Therefore, by Condition~(iii), the degrees of $ x $ and $ y $ are at least $ 3(k-2)+3(k-2)(k-2)= 3(k-1)(k-2) $. Thus, $ \mu\geq 2 $ and by \eqref{eq:vG}, $ |V(G)|\leq q+2k-4\leq 3k-5 $ which is a contradiction.\\
	
	\textbf{Case 3.} The multiplicity of all (existing) edges are either $ (k-2) $ or $ 2(k-2) $ and there is an edge with multiplicity $ 2(k-2) $.
	
	First, we prove that the degree of every vertex is either $ (k-1)(k-2) $, or $ 2(k-1)(k-2) $. Let $ x_0 $ be a vertex in $ V(G) $ and $ X $ (resp. $ Y $) be the set of all neighbors of $ x_0 $, say $ x $ where $ xx_0 $ is of multiplicity $ (k-2) $ (resp. $2(k-2)$). For the contrary, suppose that $ x_0 $ is of degree at least $ 3(k-1)(k-2) $. Then, $ \deg(x_0)= (k-2)|X|+2(k-2)|Y|\geq 3(k-1)(k-2)$. Thus, $ |X|+2|Y|\geq 3(k-1) $. On the other hand, by Fact 1, every vertex in $ Y $ has at least $ 2(k-2) $ common neighbors with $ x_0 $ and so its degree is at least $ 2(k-2)(k-2)+2(k-2)=2(k-1)(k-2) $. Thus, by \eqref{eq:vG}, we have $ |X|+|Y|+1\leq |V(G)|\leq q+2k-|Y|-2 $. Thus, $ |X|+2|Y|\leq q+2k-3\leq 3k-4 $, which is in contradiction with $ |X|+2|Y|\geq 3(k-1) $. Hence, the degree of every vertex is either $ (k-1)(k-2) $, or $ 2(k-1)(k-2) $. 
	
	Now, let $ x_0 $ be such that $ |Y|\geq 1 $. We prove that $ |Y|= 1 $. First, by the above discussion, we have $ \deg(x_0)=2(k-1)(k-2)= (|X|+2|Y|)(k-2) $ and thus, $ |X|+2|Y|= 2(k-1) $. Also, every vertex in $ Y $ has at least $ 2(k-2) $ common neighbors with $ x_0 $ and thus, $ |X|+|Y|\geq 2(k-2)+1 $. Therefore, $ |Y|\leq 1 $. Let $ Y=\{y_0\} $. Then, $ |X|=2k-4 $ and $ y_0 $ is adjacent to all vertices in $ X $.
	
	Next, let $ X_1 $ and $ X_2 $ be the set of all vertices in $ X $ of degree $ (k-1)(k-2) $ and $ 2(k-1)(k-2) $, respectively. It is evident that $ X_1 $ is nonempty (otherwise, by \eqref{eq:vG}, $ |V(G)|\leq q+2\leq k+1 $). We are going to prove that $ X_2 $ is empty. For the contrary, suppose that $ X_2\neq \emptyset $. First, note that every vertex in $ X_1 $ has $ (k-2) $ common neighbors with $ x_0 $ and so has no neighbor in $ V(G)\setminus N[x_0]$.  Also, for every vertex $ x\in X_1 $ and each of its neighbors, say $ x'\neq x_0 $, the triangle $ \{x_0,x,x'\} $ is a member of any distinct $ K_3 $-decomposition of $ G $. This implies that every vertex $ x'\in X_2 $ has at most $ k-2 $ neighbors in $ X_1 $ (because the edge $ x_0x' $ is covered by exactly $ k-2 $ distinct triangles in the $ K_3 $-decomposition). Therefore, every vertex $ x'\in X_2 $ has at most $ k-2+|X_2|-1+2=k+|X_2|-1 $ neighbors in $ N[x_0] $. On the other hand, $ x' $ has at least $ 2(k-1)-1 $ distinct neighbors (because its degree is $ 2(k-1)(k-2) $ and by the above argument, at most one of its adjacent edges could have multiplicity $ 2(k-2) $). Thus, $ x' $ has at least $ 2k-3-(k+|X_2|-1) = k-|X_2|-2 $ neighbors in $ V(G)\setminus N[x_0] $. First, suppose that all vertices in $ X_2 $ have no neighbor in  $ V(G)\setminus N[x_0] $. Thus, $ |X_2|\geq k-2 $ and every vertex in $ X_2 $ has at least $ 2k-3 $ distinct neighbors, so every vertex in $ X_2 $ is adjacent to all vertices in $ X_1 $. This yields that every vertex in $ X_1 $ has at least $ |X_2|+2\geq k $ adjacent vertices which is a contradiction (vertices of degree $ (k-1)(k-2) $ have at most $ (k-1) $ neighbors). Now, suppose that there is a vertex $ x'\in X_2  $ with a neighbor $ x'' $ in  $ V(G)\setminus N[x_0] $. By above arguments, $ x'' $ has no neighbor in $ X_1\cup Y\cup \{x_0\} $. On the other hand, $ x' $ and $ x'' $ have at least $ k-2 $ common neighbors. Therefore, $ x' $ has at least $ (k-2)-(|X_2|-1)+1=k-|X_2| $ neighbors in  $ V(G)\setminus N[x_0] $. Hence, $ |V(G)|\geq 2(k-2)+2+k-|X_2|=3k-|X_2|-2 $.  Also, by \eqref{eq:vG}, we have $ |V(G)|\leq q+2k-(|X_2|+2)\leq 3k-|X_2|-3 $, a contradiction. This shows that $ X_2 $ is empty.
	
	Hence, the vertices in $N[x_0] $ has no neighbor in $ V(G)\setminus N[x_0] $ and the induced subgraph of $ G $ on $ N[x_0] $ has exactly $ k(k-1)(k-2) $ edges. Thus, by Condition (i), the induced subgraph of $ G $ on  $ V(G)\setminus N[x_0] $ has $ {\beta}/{2} $ edges. On the other hand, due to Fact 1 and Condition (ii), this subgraph has at least $ k $ non-isolated vertices and at least $ k(k-1)(k-2)/2 $ edges which is a contradiction. This completes the proof.
\end{proof}

Here, we prove a lower bound for $ D(n,k,3) $, in the conditions of \pref{thm:upper_r0a0b!0}.

\begin{thm} \label{thm:lower_r0a0b!0}
	For every positive integer $ k\geq 3 $, there exists a positive integer $ n_0 $ such that for every $ n\geq n_0 $, if $ (n-2)\equiv 0 \pmod{k-2} $ and $ (n-1)(n-2)\equiv 0\pmod {(k-1)(k-2)} $, then $ D(n,k,3)\geq J(n,k,3)-O(k) $.
\end{thm}
\begin{proof}
	Let $ n,k $  be such that $ (n-2)\equiv 0 \pmod{k-2} $ and  $ (n-1)(n-2)\equiv 0\pmod {(k-1)(k-2)} $. Also, let $ \beta \equiv n(n-1) (n-2) \pmod {k(k-1)(k-2)} $, where $ 0\leq \beta< k(k-1)(k-2) $. Thus, $ \beta $ is a multiple of $ (k-1)(k-2) $. Let $ \beta =q (k-1)(k-2) $ for some $ 0\leq q<k $. Thus,
	\[J(n,k,3)= \frac{n(n-1)(n-2) - \beta}{k(k-1)(k-2)}. \]
	Note that if $ k $ is even, then $ n $ is also even and so $ q $ is also even. Also, $ \beta $ is always divisible by $ 3 $, so if $ k $ is divisible by $ 3 $, then $ q $ is divisible by $ 3 $ as well.
	
Now,by \pref{lem:multigraph3}, it suffices to prove that there exists a multigraph $ G $ on $ n $ vertices satisfying Conditions (ii)-(v), with  $ 2|E(G)|= {\beta + O(k) k(k-1)(k-2)}$.


Fix some positive integer $ l $ and let $ G^l $ be the multigraph $(k-2)K_{l(k-1)+1} $. It is clear that $ G ^l$ satisfies Conditions (ii), (iii) and (v). Also, we have
\[2|E(G^l)|= (k-2)(k-1)l(l(k-1)+1).\]
If $ k\equiv 1,2 \pmod{3} $, then set $ l=2 $ and if $ k\equiv 0\pmod{3} $, then set $ l=3 $. Thus, $ G^l $ satisfies the conditions of \pref{thm:triangle} (note that the degree of each vertex is equal to $ l(k-1)(k-2) $ which is always an even number). Therefore, $ G^l $ admits a distinct $ K_3 $-decomposition.  

Now, let $ t,c $ be some integers such that $ 0<t<k $ and $ c\in\{1,2\} $ and $ l(1-l)t +kc= q $ (such $ t,c $ exist, because $ \gcd(l(1-l),k) $ divides $ q $). Also, let $ G  $ be the multigraph on $ n $ vertices obtained from $ t $ disjoint copies of $ G^l $ by adding some isolated vertices.  Therefore, evidently $ G $ satisfies Conditions (ii)-(v) of \pref{lem:multigraph3} and we have
\[2|E(G)|= 2t|E(G^l)|=q(k-1)(k-2)+(tl^2-c)k(k-1)(k-2)=\beta+ O(k)k(k-1)(k-2).  \]
This completes the proof. 
\end{proof}

\begin{rem}\label{rem:equality}
Note that in the proof of  \pref{thm:lower_r0a0b!0}, if $ k\equiv 1,2 \pmod{3} $ and $ q=k-2 $, then $l=2$ and $ t=c=1 $ and so $ tl^2-c=3 $. This implies that $ D(n,k,3)\geq J(n,k,3)-3 $ and by \pref{thm:upper_r0a0b!0}, we have $ D(n,k,3)=J(n,k,3)-3 $. Hence, there are infinitely many numbers $ n,k $ for which equality holds in the upper bound given in \pref{thm:upper_r0a0b!0}. 
\end{rem}

\section{The case $ (n-2)\equiv 0 \pmod { k-2} $ and $ (n-1)(n-2)\not\equiv  0\pmod {(k-1)(k-2)} $\label{sec:3}}
Finally, as the last task, we look into the case that $ (n-2)\equiv 0 \pmod { k-2} $ and $ (n-1)(n-2)\not\equiv 0\pmod {(k-1)(k-2)} $.
Due to the recursion \eqref{eq:recursion} as well as \eqref{eq:t=2}, we have the following upper bound for $ D(n,k,3) $.
\begin{cor} \label{thm:upper_r0a!0}
	Let $ n,k $ be two positive integers such that $ n>k\geq 4 $, $ (n-2)\equiv 0 \pmod{k-2} $ and $ (n-1)(n-2)\not\equiv 0\pmod {(k-1)(k-2)} $. Then, $ D(n,k,3)\leq J'(n,k,3) $, where
	\[J'(n,k,3):=\left\lfloor\frac{n}{k} \left(\left\lfloor \frac{(n-1)(n-2)}{(k-1)(k-2)} \right\rfloor-1\right) \right\rfloor. \]
\end{cor}
Analogous to previous cases, one may expect that the difference of $ D(n,k,3) $ and $ J'(n,k,3) $ should be bounded by a function depending only on $ k $. However, as a matter of surprise, this is not the case at least for infinitely many $ k $. In fact, there are infinitely many $ k $ for which $ J'(n,k,3)-D(n,k,3) $ tends to infinity as $ n $ grows. 
This is what we prove in the following theorem.

\begin{thm}\label{thm:p2}
Let $ n,k $ be positive integers such that $ k\equiv 0\pmod{6} $, $k\geq 12$, $ (n-2)\equiv 0 \pmod { k-2} $ and $ (n-1)(n-2)\equiv 2(k-2)\pmod {(k-1)(k-2)} $. Then $ D(n,k,3)\leq J'(n,k,3)-\lfloor\frac{n}{3k^3}\rfloor $.
\end{thm}
\begin{proof}
Suppose that $ n,k $ satisfy the given conditions. For the contrary, suppose that $D(n,k,3)\geq J'(n,k,3)-\lfloor\frac{n}{3k^3}\rfloor+1$. Thus, by \pref{lem:multigraph3}, there is a multigraph $ G $ satisfying Conditions (i)-(iv). By Condition (ii), for every vertex $ x\in V(G) $, we have $ \deg_G(x)\equiv 2(k-2) \pmod{(k-1)(k-2)} $. First, note that there is no vertex $ x\in V(G) $ with $ \deg(x)=2(k-2) $. To see this, let $ y $ be a neighbor of $ x $. Since $ m_G(xy)\geq k-2 $ and $ G $ has a distinct $ K_3 $-decomposition, $ x $ and $y$ must have at least $ k-2 $ distinct common neighbors. Thus, $ \deg(x)\geq (k-1)(k-2)>2(k-2) $. Therefore, for every vertex $ x\in V(G) $, $\deg(x)\geq 2(k-2)+(k-1)(k-2)= (k+1)(k-2)$. We call a vertex of degree $ (k+1)(k-2) $ a \textit{good vertex} and other vertices as \textit{bad} ones. Also, let $ B $ be the set of all bad vertices.
Now, we prove some facts.\\

\textbf{Fact 1.} Every good vertex has exactly $ k+1 $ distinct neighbors and all its adjacent edges have  multiplicity $ k-2 $.\\

To see this, note that all edges on a good vertex $ x $ must be of multiplicity $ k-2 $ (since otherwise, $ x $ should have at least $ 2(k-2)+1 $ distinct neighbors and its degree would be at least $ (2k-3)(k-2) $ which is a contradiction for $ k>4 $), so $ x $ has exactly $ k+1 $ distinct neighbors.

For every vertex $ x\in V(G) $ and positive integer $ d $, the closed (resp. open) $ d$-neighborhood of $ x $, denoted by $ N_d[x] $ (resp. $ N_d(x) $), is defined as the set of all vertices of distance at most (resp. exactly) $ d $ from $ x $. 
In the following fact, we prove that there is at least one bad vertex in the closed $2$-neighborhood of every vertex. \\

\textbf{Fact 2.} For every vertex $ x \in V(G)$, we have $ N_2[x]\cap B\neq \emptyset $.\\

To see this, for the contrary, suppose that for $ x\in V(G) $, all vertices in $ N_2[x] $ are good. By Fact~1, every vertex in $ N_2[x] $ has exactly $ k+1 $ distinct neighbors. First, we prove that $ |N_2(x)|\leq 2 $. For this, we count $e$ the number of edges between $ N_1(x) $ and $ N_2(x) $. For every vertex $ y\in N_1(x) $, $ x$ and $y $ have at least $ k-2 $ common neighbors, so $ y $ has at most two neighbors in $ N_2(x) $. On the other hand, for every vertex $ z\in N_2(x) $ with an adjacent vertex $ y\in N_1(x) $, $ z$ and $y $ have at least $ k-2 $ common neighbors with at least $ k-3 $ of which in $ N_1(x) $. Thus, $ z $ has at least $ k-2 $ distinct neighbors in $ N_1(x) $. This shows that $(k-2) |N_2(x)| \leq e\leq  2|N_1(x)| = 2(k+1) $. This implies that $ |N_2(x)|\leq 2(k+1)/(k-2)<3 $, for $ k>8 $. Now, we prove that $ N_3(x) $ is empty. Take $ w\in N_3(x) $ with a neighbor $ z\in N_2(x) $. Then, $ w $ and $ z $ have at least $ k-2 $ common neighbors in $ N_2(x)\cup N_3(x)  $. Also, $ z $ has at least $ k-2 $ distinct neighbors in $ N_1(x) $. Therefore, $ z $ has at least $ 2k-3 $ distinct neighbors and thus $ z $ is not good, a contradiction. Hence, $ N_3(x) $ is empty. 

Let $ H $ be the connected component of $ G $ containing $ x $. Then, by the above argument, $ k+2\leq |V(H)|\leq k+4  $. Since $ H $ has a distinct $ K_3 $-decomposition, $ 2|E(H)|=|V(H)|(k+1)(k-2) $ is a multiple of $ 3 $. Thus, since $k\equiv 0\pmod 6$, we have $|V(H)|=k+3$. On the other hand, the underlying simple graph of $ H $ is $ k+1 $ regular, so $ |V(H)| $ should be even. This leads us to a contradiction with $ k\equiv 0\pmod{6} $ and thus Fact 2 is proved. 

Now, we find a lower bound for the number of edges of $ G $. Let $ F $ be the underlying simple graph obtained from $ G $. By definition, for every vertex $ x\in V(G)\setminus B $, we have $ \deg_F(x)=k+1 $. Also, for every $ x\in B $, if multiplicities of all edges incident with $x$ are equal to $k-2$,  we have $ \deg_F(x)\geq 2+2(k-1)=2k $ and otherwise, we have $\deg_F(x)\geq 2(k-2)+1=2k-3$. Therefore, 
\begin{equation}
\sum_{x\in V(G)}\deg_F(x) \geq (k+1)n+(k-4)|B|. \label{eq:f3}
\end{equation}

Let $ X $ be the set of vertices in $ V(G)\setminus B $ which have at least one neighbor in $ B $ and let $ Y= V(G)\setminus(B\cup X) $. By Fact 2, every vertex in $ Y $ has a neighbor in $ X $ and by definition, every vertex in $ X $ has a neighbor in $B$. So, $|Y|\leq |E(X,Y)|\leq k|X|$ and thus 
\begin{equation}\label{eq:f1}
|V(G)|=n\leq |B|+(k+1)|X|.
\end{equation}
Now, we have
\begin{align}
\sum_{x\in V(G)} \deg_F(x) &= 
\sum_{x\in B} \deg_F(x)+(k+1)(n-|B|) \nonumber\\
&\geq |X|+(k+1)(n-|B|) \nonumber \\
&\geq \frac{n-|B|}{k+1}+(k+1)(n-|B|)=\frac{1+(k+1)^2}{k+1}(n-|B|), \label{eq:f2}
\end{align}
where the last inequality is due to \eqref{eq:f1}. Thus, by \eqref{eq:f3} and \eqref{eq:f2}, we have
\begin{align*}
\sum_{x\in V(G)}\deg_F(x) &\geq \max\{(k+1)n+(k-4)|B|,\frac{1+(k+1)^2}{k+1}(n-|B|)\}\\
&\geq \frac{w_1((k+1)n+(k-4)|B|)+w_2(\frac{1+(k+1)^2}{k+1}(n-|B|))}{w_1+w_2},
\end{align*}
for every positive numbers $ w_1,w_2 $. Set $ w_1=1+(k+1)^2 $ and $ w_2=(k-4)(k+1) $. Then, we have
\[\sum_{x\in V(G)}\deg_F(x) \geq (k+1+\frac{k-4}{2k^2-k-2})n. \]
Hence, 
\begin{equation}\label{eq:f4}
2|E(G)|\geq 2(k-2)|E(F)|\geq (k+1+\frac{k-4}{2k^2-k-2})(k-2)n.
\end{equation}
On the other hand,
\begin{align*}
J'(n,k,3)&=\left\lfloor\frac{n}{k} \left(\frac{(n-1)(n-2)-2(k-2)}{(k-1)(k-2)} -1\right) \right\rfloor\\
&=\left\lfloor\frac{n(n-1)(n-2)-n(k+1)(k-2)}{k(k-1)(k-2)} \right\rfloor.\\
\end{align*}
Thus, by Condition (i) of \pref{lem:multigraph3}, we have
\begin{align*}
2|E(G)|&= n(n-1)(n-2)- k(k-1)(k-2) \left(J'(n,k,3)-\left\lfloor\frac{n}{3k^3}\right\rfloor+1\right) \\
&< n\left((k+1)(k-2)+\frac{(k-4)(k-2)}{2k^2-k-2}\right),
\end{align*}
which is in contradiction with \eqref{eq:f4}.
\end{proof}

Now, we prove a lower bound for $ D(n,k,3) $ in terms of $ J'(n,k,3) $.

\begin{thm} \label{thm:lower_r0a!0}
	For every positive integer $ k\geq 4 $, there exists a positive integer $ n_0 $ such that for every $ n\geq n_0 $, the following statement holds.
	
	Suppose that $ (n-2)\equiv 0 \pmod{k-2} $ and $ (n-1)(n-2)\equiv  p(k-2)\pmod {(k-1)(k-2)} $ for some integer $p$, $1\leq p\leq k-2$. Subsequently, if 
	\begin{equation}\label{eq:*}
	\tag{$*$}
	\text{ neither } k,p\equiv 4\pmod 6 \text{ nor } k, (p-2)\equiv 0 \pmod 6,
	\end{equation} 
	then, we have $ D(n,k,3)\geq J'(n,k,3)-O(k^{3+o_k(1)}) $. Otherwise, i.e. if either $k,p\equiv 4\pmod 6$ or $k,(p-2)\equiv 0 \pmod 6$, then we have  $ D(n,k,3)\geq J'(n,k,3)-O(k^{o_k(1)})n $.
	
%
\end{thm}
\begin{proof}[\rm \textbf{Proof of \pref{thm:lower_r0a!0}}]
	Let $ n,k $  be such that $ (n-2)\equiv 0 \pmod{k-2} $. Also, let  $\alpha\equiv  (n-1)(n-2) \pmod {(k-1)(k-2)} $, where $ 0<\alpha < (k-1)(k-2) $.
	This implies that $ \alpha $ is a multiple of $ (k-2) $. Let $ \alpha = p(k-2) $, where $ 0 <p< k-1 $. So, 
\begin{align*}
J'(n,k,3)&=  \left\lfloor \frac{n(n-1)(n-2)- n(p+k-1) (k-2)}{k(k-1)(k-2)} \right\rfloor \\[4pt]
&= \frac{n(n-1)(n-2)- n(p+k-1) (k-2)-\beta }{k(k-1)(k-2)},
\end{align*}
	 where $ \beta \equiv n(n-1) (n-2)- n(p+k-1) (k-2) \pmod {k(k-1)(k-2)} $, $ 0\leq \beta< k(k-1)(k-2) $. 
	 Thus, $ \beta $ is a multiple of $ (k-1)(k-2) $, and we may write $ \beta = q(k-1)(k-2) $, for some $ 0\leq q<k $.	 
	 By \pref{lem:multigraph3}, it suffices to prove that there exists a multigraph $ G $ on $ n $ vertices satisfying Conditions (i)-(v) such that  $ 2|E(G)|= {\beta + n(p+k-1)(k-2) + O(k^{3+o(1)}) k(k-1)(k-2)} $.

In order to construct such a multigraph $ G $, we do the following steps. First, we construct three multigraphs $ G_1, G_2,G_3 $ on respectively $ n_1,n_2,n_3 $ vertices with the following properties,
	\begin{itemize}
	\item[(a)]  the multiplicity of each edge in $ G_i$, $i=1,2,3 $, is exactly $ k-2 $,
	\item[(b)] the degree of each vertex in $ G_i$ is equal to $ (p+l_i(k-1))(k-2) $, $i=1,2,3$, where $l_1, l_2,l_3 $ are positive integers. Also, $l_1=1$ when \eqref{eq:*} holds.
	\item[(c)] the graph $ G_i$, $i=1,2,3 $, admits a distinct $ K_3 $-decomposition, 
\item[(d)] $ \gcd(n_1,n_2) $ divides $ n-n_3 $,
\item[(e)]  $n_1=O(k)$ when \eqref{eq:*} holds and $n_1=O(k^2)$, otherwise. Also, $n_2= O(k^{2+o_k(1)})$, $ n_3= O(k^{3+o_k(1)}) $, $l_1, l_2,l_3=O(k^{1+o_k(1)})$, and
\item[(f)] $ n_i (l_i-1) \equiv 0 \pmod k $, $i=1,2$ and $ n_3 (l_3-1) \equiv q\pmod k $.
\end{itemize}

Note that, if $ n $ is sufficiently larger than $ n_1,n_2,n_3 $, then by (d), there exist nonnegative integers $ r_1,r_2 $ such that $ r_2< n_1 $ and $ r_1n_1+r_2n_2=n-n_3 $. Now, let $ G $ be the disjoint union of $ r_1 $ copies of $ G_1 $, $ r_2 $ copies of $ G_2 $ and one copy of $ G_3 $. Properties (a), (b) and (c) yield that $ G $ is a multigraph on $ n $ vertices satisfying Conditions (ii)-(v) of \pref{lem:multigraph3}. On the other hand, by (b), we have

\begin{align}
2|E(G)|&= (r_1n_1 (p+l_1(k-1))+ r_2 n_2(p+l_2(k-1))+n_3 (p+l_3(k-1)))(k-2) \nonumber  \\ 
&= n (p+k-1) (k-2)+ (r_1n_1(l_1-1)+r_2 n_2 (l_2-1)+n_3 (l_3-1)) (k-1)(k-2) \nonumber \\
&= \beta+  n (p+k-1) (k-2)+  (r_1\frac{n_1(l_1-1)}{k}+r_2 \frac{n_2 (l_2-1)}{k}+\frac{n_3 (l_3-1)-q}{k}) k(k-1)(k-2). \label{eq:edges}
\end{align}
Due to (f), the coefficient of $ k(k-1)(k-2) $ is an integer. Also, when \eqref{eq:*} holds, we have $l_1=1$ and $r_2<n_1=O(k)$. Thus, in this case, by (e), we have 
\[2|E(G)|= \beta + n(p+k-1)(k-2) + O(k^{3+o_k(1)})\, k(k-1)(k-2).\]
Otherwise, when \eqref{eq:*} does not hold, by (e) we have 
\[2|E(G)|= \beta + n(p+k-1)(k-2) + O(k^{o_k(1)})n\, k(k-1)(k-2),\]
and by \pref{lem:multigraph3} we are done. The rest of the proof is devoted to the construction of the graphs $ G_1,G_2,G_3 $ in several cases.

We consider all possibilities of $ k $ and $ p $ modulo $ 3 $. First, note that if $ k\equiv 1 \pmod 3 $, then $ (n-1)(n-2)\equiv p(k-2) \pmod 3 \equiv -p \pmod 3 $, so $ p\not\equiv 2 \pmod 3 $. Thus, we have to investigate 8 possibilities. Moreover, $ \alpha $ is always even, so if $ k $ is odd, then $ p $ is even.

In construction of the graphs $ G_1,G_2,G_3 $, we use distinct $ K_3 $-decompositions of the graphs obtained from simple GDDs. Fix positive integers $ g,l $ such that $ g $ divides $ p+l(k-1) $ and define $ u= (p+l(k-1))/g+1 $. Consider the complete multipartite graph on $ gu $ vertices with $ u $ parts of equal size $ g $ and replace each edge with $ k-2 $ parallel edges. Let us denote the obtained graph by $ G^{g,l} $. It is clear that $ |V(G^{g,l})|= gu= p+l(k-1)+g $ and if a simple $ (3,k-2) $-GDD($ g^u $) exists, then $ G^{g,l} $ admits a distinct $ K_3 $-decomposition. 
Therefore, in the light of \pref{cor:simpleGDD}, the graph $ G^{g,l} $ can be chosen as a candidate for the graphs $ G_1,G_2,G_3 $ satisfying Conditions (a)-(c), provided the following conditions hold.


\begin{itemize}
	\item[(a$ ' $)] $ p+l(k-1) \equiv 0 \pmod g $,
	\item[(b$ ' $)] $ (k-2)(p+l(k-1))(p+l(k-1)+g)\equiv 0 \pmod 3 $,
	\item[(c$ ' $)] $p+l(k-1)\geq 2g$ and $ k-2\leq g(u-2)=p+l(k-1)-g $.
\end{itemize}
Note that if $ k $ is odd, then $ p $ is even and so $ (k-2)(p+l(k-1)) $ (the degree of each vertex in $ G^{g,l} $) is always even.
\paragraph{Construction of $ G_3 $.}

Let $ p_1,\ldots,p_r $ and $ p_{r+1},\ldots,p_{r+s} $ be respectively the distinct prime powers except powers of $2$ and $3$ in $ k $ and $ k-1 $. Let $ p'\in\{0,2,4\} $ be such that $ p'\equiv p\pmod{3}$. Applying \pref{lem:number},  there exists some integer $ g_3=O(k^{1/\Omega({\log \log k)}}) $, satisfying the following conditions,

\begin{align}
&g_3\not \equiv 0\pmod{p_i},\forall\ 1\leq i\leq r+s \nonumber\\
&g_3\not \equiv p'-p\pmod{p_i},\forall\ 1\leq i\leq r\nonumber	\\
&g_3\equiv 1\pmod{2} \text{ and } g_3\not\equiv p' -p\pmod{4} \nonumber \\
& g_3 \equiv 2 \pmod{3} \label{eq:g_3}.
\end{align}
Therefore, we have

\begin{align}
 \gcd(g_3,k)&=1, \label{eq:A1} \\
  \gcd(g_3,k-1)&=1, \label{eq:A2} \\
 \gcd(p-p'+g_3,k)&\in\{1,2\}. \label{eq:A3}
\end{align}
	
Now, by Chinese Remainder Theorem, there exists some integer $ l_3 $, $g_3\leq l_3=O(g_3k)  $, such that  
\begin{align}
&l_3\equiv p' \pmod{k},\label{eq:l_3}\\
&p+l_3 (k-1) \equiv 0 \pmod{g_3}, \nonumber \\
&\text{if } k \equiv 2\pmod 3 \text{ and } p\equiv 1\pmod 3, \text{then }  l_3\equiv 1\pmod 3 . \label{eq:l_33}
\end{align}
(Note that such $ l $ exists since \eqref{eq:A1} and \eqref{eq:A2} hold.)

The obtained $ g_3,l_3 $ clearly satisfy Conditions (a$'$) and (c$'$). Now, we check that Condition (b$'$) holds. If $ k\equiv 2 \pmod{3} $, there is nothing to show. If $ k\equiv 0\pmod{3} $, then $ l_3\equiv p'\equiv p\pmod{3} $. Thus, $ p+l_3(k-1)\equiv 0\pmod{3} $. Finally, if $ k\equiv 1\pmod{3} $, then $ p\not \equiv 2\pmod 3 $. If $ p \equiv 0\pmod{3}  $, then again $ p+l_3(k-1)\equiv 0\pmod{3} $ and if $ p \equiv 1\pmod{3}  $, then since $ g_3\equiv 2\pmod{3} $, we have $  p+l_3(k-1)+g_3\equiv 0\pmod{3}  $. Therefore, (b$'$) holds and hence, the graph $ G^{g_3,l_3} $ exists and admits a $K_3$-decomposition. Let $ G_3 $ be the disjoint union of $ t $ copies of $ G^{g_3,l_3} $, where $ t $ is an integer which is determined as follows.  Clearly, $ G_3 $ satisfies Conditions (a)-(c). We are going to determine $ t $ such that Condition (f) holds. Note that $ n_3= |V(G_3)|= t (p+l_3(k-1)+g_3) $. Thus,
\[n_3(l_3-1)\equiv t (p+l_3(k-1)+g_3)(l_3-1)\equiv t(p-p'+g_3)(p'-1) \pmod{k}.  \]
Now, note that if $k  $ is even, then $ n, \beta/(k-2) $ and so $ q $ are even as well. Also, if 
$ k\equiv 0 \pmod 3$ and $ p\equiv 1 \pmod 3 $, we have $ \beta \equiv 0\pmod 3 $ and so $ q \equiv 0\pmod 3 $. Also, $ p'\in\{0,2,4\} $. This along with \eqref{eq:A3} ensure that $ \gcd((p-p'+g_3)(p'-1),k) $ divides $ q $. Therefore, there exists some integer $t $,  $1\leq t\leq k $, such that $ t(p-p'+g_3)(p'-1)\equiv q \pmod{k} $ and thus (f) holds.  

It should be noted that $ l_3=O(g_3k)=O(k^{1+o_k(1)}) $ and $ t=O(k) $. Thus, $ n_3=t (p+l_3(k-1)+g_3) = O(k^{3+o_k(1)})$. Moreover, in order to verify (d), we will need the following two facts about $ n_3 $.  \\

\textbf{Fact 1. } If either $ k\equiv 1 \pmod 3 $ and $ p\equiv 1 \pmod 3 $, or $ k \equiv 0 \pmod 3$ and $ p\not \equiv 1 \pmod 3 $, or $ k\equiv 2 \pmod 3 $ and $ p\equiv 1 \pmod 3  $, then $ n \equiv n_3 \pmod 3 $.

To see this, first suppose that $ k\equiv 1 \pmod 3 $ and $ p\equiv 1 \mod 3 $. Then, $ (n-1)(n-2) \equiv p(k-2) \pmod 3 \equiv 2 \pmod 3 $. Thus, $ n\equiv 0 \pmod 3 $. Also, by \eqref{eq:g_3}, $g_3\equiv 2 \pmod{3}$. Therefore, $ n_3=t (p+l_3(k-1)+g_3) \equiv 0\pmod{3} $. 

Now, suppose that $ k \equiv 0 \pmod 3$ and $ p\not \equiv 1 \pmod 3 $. Then, $2q\equiv \beta \equiv 2n(p-1) \pmod 3  \pmod 3 $. So, $ q\equiv n(p-1) \pmod 3 $.  On the other hand, 
by (f), we have $ n_3(l_3-1)\equiv q \pmod 3 $ and by \eqref{eq:l_3}, $ l_3\equiv p'\equiv p\pmod{3} $. Thus, 
$ n_3(p-1)\equiv q \equiv n(p-1) \pmod 3 $. Now since $ p\not \equiv 1 \pmod 3 $, we have $ n \equiv n_3 \pmod 3 $. 

Finally, suppose that  $ k\equiv 2 \pmod 3 $ and $ p\equiv 1 \mod 3  $. Then, $ n\equiv 2 \pmod 3 $. 
Also, by \eqref{eq:g_3} and \eqref{eq:l_33}, $ n_3= t(p+l_3(k-1)+g_3)\equiv t\pmod 3 $. If $ n_3\equiv 2\pmod 3 $, we are done. Suppose that $ n_3\not \equiv 2\pmod 3 $. So, we may replace $ t $ with either $ t+k $ or $ t+2k $ such that $ n_3\equiv t\equiv 2\pmod 3  $.  \\

\textbf{Fact 2. } If $ k $ and $ p $ are both even, then $ n \equiv n_3 \pmod 2 $. 

To see this, note that since $ k $ is even, $ n $ and so $ q $ are also even. Further, by \eqref{eq:l_3} and since $p'\in\{0,2,4\}  $, we have $ l_3 $ is even as well. Thus, by (f), $ n_3 $ is also even.

\paragraph{Construction of $ G_1$.}
If either $ k\equiv 2 \pmod 3 $, or $ k\equiv 1 \pmod 3 $ and $ p\equiv 0 \pmod 3  $, or $ k\equiv 0 \pmod 3 $ and $ p\not\equiv 2 \pmod 3 $, then $ (g,l)=(1,1) $ satisfies Conditions (a$ ' $)-(c$ ' $), thus, set $ G_1= G^{1,1} $. 

Now, suppose that either $ k,p\equiv 1\pmod 3 $, or $ k\equiv 0\pmod 3 $ and $ p\equiv 2\pmod 3 $.
If $ p+k $ is odd, then $ (g,l)=(2,1) $ satisfies Conditions (a$ ' $)-(c$ ' $), thus, set $ G_1= G^{2,1} $. Now, assume that $ p+k $ is even (i.e. both $k$ and $p$ are even).  

If $k,p\equiv 1 \pmod 3$, then set $g_1=2$ and $l_1=2d$ such that $1\leq d\leq k$ is an integer where $p-2d+2\equiv 0\pmod k$ (such $d$ exists because both $p-2d+2$ and $k$ are even). It is evident that $(g,l)=(g_1,l_1)$ satisfies Conditions (a$'$)-(c$'$). Now, set $G_1=G^{g_1,l_1}$. Also, $l_1=O(k)$ and $n_1=l_1(k-1)+p+2=O(k^2)$ and (e) holds. Finally, $n_1\equiv p-2d+2\pmod k\equiv 0\pmod k$ and so, (f) holds.

Analogously, if $k\equiv 0\pmod 3$ and $p\equiv 2 \pmod 3$, then set $g_1=1$ and $l_1=3d$ such that $1\leq d\leq k$ is an integer where $p-3d+1\equiv 0\pmod k$ (such $d$ exists because both $p-3d+1$ and $k$ are divided by $3$). It is evident that $(g,l)=(g_1,l_1)$ satisfies Conditions (a$'$)-(c$'$). Now, set $G_1=G^{g_1,l_1}$. Also, $l_1=O(k)$ and $n_1=l_1(k-1)+p+1=O(k^2)$ and (e) holds. Finally, $n_1\equiv p-3d+1\pmod k\equiv 0\pmod k$ and so, (f) holds.

\paragraph{Construction of $G_2 $.}
Now, we describe the construction of $G_2$ in the following two cases. \\

\textbf{Case 1.}  Suppose that either $ k\equiv 2\pmod 3  $ or $ k\equiv 1\pmod 3 $ and $ p\equiv  0\pmod 3 $, or $ k\equiv 0 \pmod 3 $ and $ p\equiv 1 \pmod 3 $. \\

\textbf{Subcase 1.2.} $p+k$ is odd.

In this case, $ (g,l)=(2,1) $ satisfies Conditions (a$ ' $)-(c$ ' $). Set $ G_2= G^{2,1} $. Thus, $ \gcd(n_1,n_2)=\gcd(p+k,p+k+1)=1 $. So, Conditions (d)-(f) obviously hold. \\

\textbf{Subcase 1.1. } $ p+k $ is even (i.e. $ p $ and $ k $ are both even).

Set $ g_2=3 $ and let $ d\in\{0,1,2\} $ be such that $ p+d(k-1)\equiv 0\pmod 3 $. Let $ p_1,\ldots, p_t $ be all distinct prime factors of $ n_1/\gcd(n_1,k) $ except $ 2$ and $3 $. Also, let $ l_2=3x+d $, where $ x =O(k^{1+o_k(1)}) $ is a positive integer satisfying the following conditions (which exists due to \pref{lem:number}). 

\begin{align}
&l_2-1=3x+d-1\equiv 0\pmod {2k},\label{eq:l_21}\\
&n_2=(k-1)3x+(k-1)d+p+3 \not \equiv 0\pmod {p_i}\quad \forall 1\leq i\leq t,\label{eq:l_22}\\
&\text{if }\  k\equiv 2 \pmod 3 , \text {then }\ n_2=(k-1)3x+(k-1)d+p+3 \equiv 3\pmod {9}. \label{eq:l_23}
\end{align} 

(Note that such number $ x $ exists because if $ p_i $ divides $ k-1 $, then $ p_i $ divides $ n_1-k+1=p+1 $. So, $ p_i $ does not divide $ p+3 $).

Now, we prove that $ p+l_2(k-1)\equiv 0\pmod 3 $. It is obvious when $ k\equiv 1\pmod 3 $ and $ p\equiv 0\pmod 3 $. If $ k\equiv 2\pmod 3 $, then by \eqref{eq:l_23}, $ n_2=p+l_2(k-1)+3\equiv 0\pmod 3 $. If $ k\equiv 0\pmod 3 $ and $ p\equiv 1\pmod 3 $, then by \eqref{eq:l_21}, $ l_2\equiv 1 \pmod 3 $ and thus, $ p+l_2(k-1)\equiv 0\pmod 3 $. Hence, $(g,l)=(g_2,l_2)$ satisfies Conditions (a$'$)-(c$'$).

Now, we prove that (d) holds. Let $ \hat{p}\neq 2,3 $ be a prime factor of $ n_1=p+k $. If $ \hat{p} $ does not divide $ \gcd({n_1,k}) $, then by \eqref{eq:l_22}, $ \hat{p} $ does not divide $ n_2$. If $ \hat{p} $ divides $ \gcd({n_1,k}) $, then $ \hat{p} $ divides $ p $ and by \eqref{eq:l_21}, $ l_2\equiv 1\pmod {\hat{p}} $. Thus, $ n_2= (k-1)l_2+p+3 \equiv 2\pmod {\hat{p}} $ and so, $ \hat{p} $ does not divides $ n_2 $. Also, by \eqref{eq:l_23}, if $k\equiv 2 \pmod{3}$, then $ 9 $ does not divides $ n_2 $. Finally, since $ l_2\equiv 1 \pmod 4 $, $ n_2-n_1= (k-1)l_2+p+3-k-p\equiv 2 \pmod 4 $. Hence, if $ k\equiv 2\pmod 3 $ and $ p\equiv 1\pmod 3 $, then $ \gcd(n_1,n_2)=6 $ and otherwise $ \gcd(n_1,n_2)=2 $. Hence, by Facts 1,2, $ \gcd(n_1,n_2) $ divides $ n-n_3 $ and (d) holds.  

Finally, $l_2=O(k^{1+o_k(1)})$ and $n_2=(k^{2+o_k(1)})$ and thus, (e) holds. Moreover, by \eqref{eq:l_21}, $l_2-1\equiv 0\pmod k$ and thus, (f) holds. \\

\textbf{Case 2. } Now, suppose that either $ k,p\equiv 1 \pmod 3 $, or $ k\equiv 0 \pmod 3 $ and $ p\not \equiv 1 \pmod 3 $. \\

Let $ p_1,\ldots, p_r$, $ p_{r+1},\ldots,p_{r+s} $ and $ p_{r+s+1},\ldots,p_{r+s+t} $ be the prime powers except powers of $ 2,3 $ in $ n_1,k $ and $ k-1 $, respectively. 
Now, by \pref{lem:number}, there exists a positive integer $ g_2=O(k^{o_k(1)}) $, such that 
\begin{align}
&g_2\not\equiv 0 \pmod{p_i},\ 1\leq i\leq r+s+t \nonumber \\
&g_2\not \equiv -p \pmod{p_i},\ 1\leq i\leq r \nonumber \\
&g_2\not \equiv -p+1 \pmod{p_i},\ 1\leq i\leq r \nonumber\\
& g_2 \equiv 4-k-p \pmod 9\label{eq:g_24} \\
& g_2\equiv 1\pmod 2 \text{ and }  g_2\not \equiv  1-k-p \pmod 4 \label{eq:g_25}
\end{align}

By the last two conditions, we have $ \gcd(g_2,6)=1 $. Thus, we have

\begin{align}
 \gcd(g_2,k)&=1, \label{eq:i} \\
\gcd(g_2,k-1)&=1, \label{eq:ii} \\
\gcd(g_2,n_1)&=1, \label{eq:iii} \\
\gcd(g_2+p,n_1)&=2^a3^b,\text{ for some integers }  a,b, \label{eq:iv} \\
\gcd(g_2+p-1,n_1)&=2^{a'}3^{b'}, \text{ for some integers } a',b'. \label{eq:v}
\end{align}

Now, let $ d $, $ 1\leq d\leq g_2 $, be an integer such that $ p+d(k-1)\equiv 0 \pmod {g_2} $ (which exists because of \eqref{eq:ii}). Also, let $ l_2= xg_2+d $, where positive integer $ x $ will be determined shortly. Define $ G_2= G^{g_2,l_2} $. Then, $ n_2= p+l_2(k-1)+g_2= g_2(k-1)x+d(k-1)+g_2+p $. We are going to choose $ x $ such that $ l_2-1 \equiv 0\pmod {\lcm(k,36)} $ and $ \gcd(n_1,n_2)\in\{1,3\} $. 
To do this, let $ p_1,\ldots, p_t $ be all distinct prime factors of $ n_1/ \gcd(n_1,k(k-1)) $ other than $ 2,3 $. Note that by \pref{lem:number}, there exists an integer $ x=O(k^{1+o_k(1)}) $ satisfying the following.

\begin{align}
&l_2-1=g_2x+d-1 \equiv 0 \pmod {\lcm(36,k)},\label{eq:l21}\\
&n_2=g_2(k-1)x+d(k-1)+g_2+p \not\equiv 0 \pmod {p_i}, \ 1\leq i \leq t. \label{eq:l22}
\end{align}

(Such $ x $ exists because $ \gcd(g_2,6kn_1)=1 $ and $ \gcd(p_i,k-1 )=1$, $ 1\leq i\leq t $.)

Note that by the choice of $l_2$,  the pair $(g,l)=(g_2,l_2)$ satisfies Condition (a$'$). Also, $ n_2= l_2(k-1)+p+g_2 \equiv k-1+p+g_2\equiv 0 \pmod 3 $. Thus, Condition (b$'$) holds as well. Moreover, $l_2\geq g_2$ and thus, (c$'$) holds.

Now, we prove that $ \gcd(n_1,n_2) $ divides $ n-n_3 $. Let $ p'\neq 2,3 $ be a prime factor of $ n_1 $. If $ p' $ is not a prime factor of both $ k $ and $k-1$, then by \eqref{eq:l22}, $ p' $ is not a prime factor of $ n_2 $. Now, suppose that  $ p'$ is a prime factor of $ \gcd(n_1,k) $. Then 
$ n_2 = l_2(k-1)+g_2+p \equiv g_2+p-1 \pmod {p'} \not\equiv 0 \pmod {p'}  $ (because of \eqref{eq:v}).
Finally, suppose that  $ p'$ is a prime factor of $ \gcd(n_1,k-1) $. Then, $n_2 = l_2(k-1)+g_2+p \equiv g_2+p \pmod {p'} \not\equiv 0 \pmod {p'}  $ (because of \eqref{eq:iv}).
This implies that $ p' $ does not divide $ \gcd(n_1,n_2) $. 
Also, $ n_2 = l_2(k-1)+g_2+p \equiv k-1+g_2+p \pmod 9 \not\equiv 0 \pmod 9  $ (because of \eqref{eq:g_24}). Similarly, due to \eqref{eq:g_25}, $n_2\not\equiv 0\pmod 4$. 
Moreover, note that $ l_2 $ and $ g_2 $ are odd, so the parity of $ n_2 $ and $p+k$ is the same. Therefore, if $p+k$ is odd, then $\gcd(n_1,n_2)\in \{1,3\} $ and if $p+k$ is even, then $\gcd(n_1,n_2)\in \{1,2,3,6\}$. Hence, by Facts 1 and 2, $ \gcd(n_1,n_2) $ divides $ n-n_3 $ and (d) holds. 

Finally, since $ l_2=xg_2+d $ and $ n_2=l_2(k-1)+p+g_2 $, we have $ l_2=O(k^{1+o_k(1)}) $ and $ n_2=O(k^{2+o_k(1)}) $ and hence (e) holds. Moreover, due to \eqref{eq:l21}, (f) holds. This completes the proof. 
 	\end{proof}

\begin{rem} \label{rem:equality2}
	It is noteworthy that for every integer $k\geq 8$, $k\neq 15$, there are infinitely many $n$ such that $(n-2)\equiv 0 \pmod{k-2}$ and $(n-1)(n-2)\not\equiv 0 \pmod{(k-1)(k-2)}$ and we have $D(n,k,3)=J'(n,k,3)$. To see this, let $\beta$ and $p$ be as in the proof of \pref{thm:lower_r0a!0} and note that for the following cases 
\begin{align*}
& k\equiv 2\pmod 3,\ p\equiv 2 \pmod 3, or \\
& k\equiv 1\pmod 3,\ p\equiv 0 \pmod 3, or \\
& k\equiv 0\pmod 3,\ p\equiv 1 \pmod 3, 
\end{align*}
the multigraphs $G_1=G^{g_1,l_1}$ and $G_2=G^{g_2,l_2}$ with $(g_1,l_1)=1$ and $(g_2,l_2)=(2,1)$ whenever $p+k$ is odd and $(g_2,l_2)=(3,1)$ whenever $p+k$ is even, satisfy Conditions (a)-(f). If, in addition, $\beta=0$, then we may assume that $G_3$ is empty. Thus, in the computations of \eqref{eq:edges}, we have $2|E(G)|= n(p+k-1)(k-2)$ and by \pref{lem:multigraph3}, we have $D(n,k,3)=J'(n,k,3)$ when $n$ is sufficiently large. So, it remains to prove that there are infinitely many $n$ such that $\beta =0$ and $p$ satisfies the above condition. To see this, let $n_0=a_1(k-1)(k-2)+a_0(k-2)+2$ such that $a_0(1-a_0) \equiv p \pmod{k-1}$, $1\leq p\leq k-2$, and $p$ satisfies the above condition (for instance whenever $k\equiv 2\pmod 3$, set $a_0=2$, whenever $k\equiv 1\pmod 3$, set $a_0=3$ and whenever $k\equiv 0\pmod 3$, let $a$ be the smallest integer satisfying $(a-1)(a-2) \leq k-1< a(a-1)$ and set either $a_0= a$ or $a_0=a+1$). Also, suppose that $a_1(k-1)+a_0\equiv 1 \pmod k$. Then, $n_0\equiv 0\pmod k$. Now, we may choose sufficiently large $n$ such that $n\equiv n_0 \pmod {k(k-1)(k-2)}$ and thus, $\beta=0$.
\end{rem}

\section{Concluding Remarks} \label{sec:conc}

In this paper, we gave some upper and lower bounds for $D(n,k,3)$. In \eqref{eq:bo1} and \eqref{eq:bo2}, the upper and lower bounds differ by at most $O(k^{3+\epsilon})$. One may see that there exist some cases where the upper bound is not met (for instance for $n\equiv 8,11 \pmod{ 60} $, we have $D(n,5,3)\leq J'(n,5,3)-1$). Nevertheless, dependency of the lower bound on $k$ remains as an open problem. On the other hand, in \eqref{eq:bo3}, the upper and lower bounds differ by at most $O(k^2)n$. Although we proved in \pref{thm:p2} that dependency of the lower bound on $n$ is necessary for $p=2$, we believe that $p=2$ is the only exception and the lower bound does not depend on $n$ for $p\neq 2$. So, we raise the following conjecture. 
\begin{con}\label{con:1}
There is a constant $c$ and for every integer $k$ there is an integer $n_0$ such that the following statement holds for every $n\geq n_0$. 

Suppose that $ (n-2)\equiv 0 \pmod {k-2} $ and $ (n-1)(n-2)\equiv  p(k-2)\pmod {(k-1)(k-2)} $, for some integer $p$, $1\leq p\leq k-2$. If $p\neq 2$, then $D(n,k,3)\geq J'(n,k,3)-O(k^c)$.  
\end{con}

We proved the above conjecture for all cases but two cases $k,p\equiv 4\pmod 6$ and $k, (p-2)\equiv 0 \pmod 6$. For the remained cases, it suffices to construct the graph $G_1$ (defined in the proof of \pref{thm:upper_r0a!0}) where we were unable to build in general. 
This gives rise to the following conjecture which implies \pref{con:1}.

\begin{con}\label{con:2}
Suppose that $k,p$ are two positive integers such that $4\leq p\leq k-4$ and either $k,p\equiv 4\pmod 6$ or $k, (p-2)\equiv 0 \pmod 6$. Then, there exists a multigraph $G$ on $O(k^c)$ vertices for some constant $c$ such that the multiplicity of each edge is equal to $k-2$, the degree of each vertex is equal to $ (p+k-1)(k-2) $ $($i.e. its underlying simple graph is $(p+k-1)$-regular$)$ and also $G$ admits a distinct $ K_3 $-decomposition.
\end{con}

The authors have proved \pref{con:2} for $k,p\equiv 4\pmod 6$ when $p=O(\sqrt{k})$ \cite{javadi}. \\
It is also noteworthy that the methods proposed in this paper can also be applied in studying $t$-covering and $t$-packing problems for general $t$ which might be the object of future work.\\


%
\providecommand{\bysame}{\leavevmode\hbox to3em{\hrulefill}\thinspace}
\providecommand{\MR}{\relax\ifhmode\unskip\space\fi MR }
\providecommand{\MRhref}[2]{%
	\href{http://www.ams.org/mathscinet-getitem?mr=#1}{#2}
}
\providecommand{\href}[2]{#2}

\end{document}